\documentclass[12pt]{amsart}

\usepackage{graphicx}
\usepackage{amssymb}
\usepackage{amsthm}
\usepackage{amsmath}
\usepackage{stmaryrd}
\usepackage{cite}
\usepackage[mathscr]{eucal}
\usepackage{hyperref}
\usepackage[margin=1.0in]{geometry}
\usepackage{comment}

\usepackage{url}
\usepackage{hyperref}

\title[The Bochner Technique]{New curvature conditions for the Bochner Technique}
\author{Peter Petersen and Matthias Wink}
\address{Department of Mathematics, UCLA, 520 Portola Plaza, Los Angeles, CA, 90095}

\email{petersen@math.ucla.edu}
\email{wink@math.ucla.edu}

\begin{document}
\newcommand{\diam} {\operatorname{diam}}
\newcommand{\Scal} {\operatorname{Scal}}
\newcommand{\scal} {\operatorname{scal}}
\newcommand{\Ric} {\operatorname{Ric}}
\newcommand{\Hess} {\operatorname{Hess}}
\newcommand{\grad} {\operatorname{grad}}
\newcommand{\Sect} {\operatorname{Sect}}
\newcommand{\Rm} {\operatorname{Rm}}
\newcommand{ \Rmzero } {\mathring{\Rm}}
\newcommand{\Rc} {\operatorname{Rc}}
\newcommand{\Curv} {S_{B}^{2}\left( \mathfrak{so}(n) \right) }
\newcommand{ \tr } {\operatorname{tr}}
\newcommand{ \id } {\operatorname{id}}
\newcommand{ \Riczero } {\mathring{\Ric}}
\newcommand{ \ad } {\operatorname{ad}}
\newcommand{ \Ad } {\operatorname{Ad}}
\newcommand{ \dist } {\operatorname{dist}}
\newcommand{ \rank } {\operatorname{rank}}
\newcommand{\Vol}{\operatorname{Vol}}
\newcommand{\dVol}{\operatorname{dVol}}
\newcommand{ \zitieren }[1]{ \hspace{-3mm} \cite{#1}}
\newcommand{ \pr }{\operatorname{pr}}
\newcommand{\diag}{\operatorname{diag}}
\newcommand{\Lagr}{\mathcal{L}}
\newcommand{\av}{\operatorname{av}}
\newcommand{ \floor }[1]{ \lfloor #1 \rfloor }
\newcommand{ \ceil }[1]{ \lceil #1 \rceil }

\newtheorem{theorem}{Theorem}[section]
\newtheorem{definition}[theorem]{Definition}
\newtheorem{example}[theorem]{Example}
\newtheorem{remark}[theorem]{Remark}
\newtheorem{lemma}[theorem]{Lemma}
\newtheorem{proposition}[theorem]{Proposition}
\newtheorem{corollary}[theorem]{Corollary}
\newtheorem{assumption}[theorem]{Assumption}
\newtheorem{acknowledgment}[theorem]{Acknowledgment}
\newtheorem{DefAndLemma}[theorem]{Definition and lemma}

\newenvironment{remarkroman}{\begin{remark} \normalfont }{\end{remark}}
\newenvironment{exampleroman}{\begin{example} \normalfont }{\end{example}}

\newcommand{\R}{\mathbb{R}}
\newcommand{\N}{\mathbb{N}}
\newcommand{\Z}{\mathbb{Z}}
\newcommand{\Q}{\mathbb{Q}}
\newcommand{\C}{\mathbb{C}}
\newcommand{\F}{\mathbb{F}}
\newcommand{\X}{\mathcal{X}}
\newcommand{\D}{\mathcal{D}}
\newcommand{\Cont}{\mathcal{C}}

\renewcommand{\labelenumi}{(\alph{enumi})}
\newtheorem{maintheorem}{Theorem}[]
\renewcommand*{\themaintheorem}{\Alph{maintheorem}}
\newtheorem*{theorem*}{Theorem}
\newtheorem*{corollary*}{Corollary}
\newtheorem*{remark*}{Remark}
\newtheorem*{example*}{Example}
\newtheorem*{question*}{Question}

\begin{abstract}
We prove a vanishing and estimation theorem for the $p^{\text{th}}$-Betti number of closed $n$-dimensional Riemannian manifolds with a lower bound on the average of the lowest $n-p$ eigenvalues of the curvature operator. This generalizes results due to D. Meyer, Gallot-Meyer, and Gallot. For example, in dimensions $n=5,6$ we obtain vanishing of the Betti numbers provided that the curvature operator is $3$-positive. As B{\"o}hm-Wilking observed, $3$-positivity of the curvature operator is not preserved by the Ricci flow.
\end{abstract}

\maketitle

\section*{Introduction}

A fundamental theme in Riemannian geometry is to understand the relationship between the curvature and the topology of a Riemannian manifold. The purpose of this paper is to prove the following vanishing and estimation theorem for the Betti numbers:

\begin{maintheorem}
Let $n \geq 3$ and let $(M,g)$ be a closed connected $n$-dimensional Riemannian manifold. Fix $1 \leq p \leq \floor{\frac{n}{2}}$ and consider the eigenvalues $\lambda_1 \leq \ldots \leq \lambda_{ {\binom{n}{2}} }$ of the curvature operator of $(M,g).$ \vspace{2mm}

If $\lambda_1 + \ldots + \lambda_{n-p} > 0$, then the Betti numbers $b_p(M)$ and $b_{n-p}(M)$ vanish. \vspace{2mm}

Furthermore, let $\kappa \leq 0$, $D>0$, and suppose that 
\begin{equation*}
\frac{\lambda_1 + \ldots + \lambda_{n-p}}{n-p}  \geq \kappa.
\end{equation*}

In case $\kappa =0$ all harmonic $p$-forms are parallel. When $\kappa \leq 0$ and $\diam M \leq D$, then there is a constant $C \left(n, \kappa D^2 \right)>0$ such that
\begin{equation*}
b_p(M) \leq \binom{n}{p} \exp \left( C \left(n,\kappa D^2 \right) \cdot \sqrt{-\kappa D^2 p( n-p ) } \right).
\end{equation*}
In particular, there exists $\varepsilon(n) > 0$ such that $\kappa D^2 \geq - \varepsilon(n)$ implies $b_p(M) \leq \binom{n}{p}.$
\label{BettiNumberEstimate}
\end{maintheorem}

Recall that the curvature operator of a Riemannian manifold is called $l$-positive if the sum of its  lowest $l$ eigenvalues is positive.

\begin{corollary*}
Let $n \geq 3$ and let $(M,g)$ be a closed $n$-dimensional Riemannian manifold. If the curvature operator is  $\ceil{\frac{n}{2}}$-positive, then $b_p(M) = 0$ for $0 < p < n.$
\end{corollary*}

The proof of Theorem \ref{BettiNumberEstimate} relies on the Bochner technique, which goes back to Bochner \cite{BochnerVectoreFieldsAndRic} who proved that the first Betti number of compact manifolds with positive Ricci curvature vanishes. For a more detailed account of the early developments of the Bochner technique the reader is referred to Yano-Bochner \cite{YanoBochnerCurvAndBetti}.
 
The theme of establishing vanishing results for the Betti numbers was continued by Berger \cite{BergerTwoFormsVanishCurvOperator} and D. Meyer \cite{DMeyerCurvOpPos} for manifolds with positive curvature operator. Micallef-Wang \cite{MicallefWangNIC} proved that the second Betti number of even dimensional manifolds with positive isotropic curvature vanishes and Dussan-Norohan \cite{DussanNoronhaNICandPureCurv} obtained a vanishing result for the Betti numbers of manifolds with nonnegative isotropic curvature provided that the curvature tensor is pure.

Using different techniques Micallef-Moore \cite{MicallefMoorePIC} proved that simply connected compact manifolds with positive isotropic curvature are homotopy spheres. \vspace{2mm}

The Ricci flow has been used extensively to obtain classification results, which in particular imply Bochner vanishing-type theorems. For example, Hamilton \cite{Hamilton3DimRF}, \cite{Hamilton4DimRFposCurvOp}, Chen \cite{ChenQuarterPinching} and B\"ohm-Wilking \cite{BW2} showed that manifolds with positive, in fact $2$-positive, curvature operators are space forms. Brendle-Schoen \cite{BrendleSchoenSphereTheorem} and Brendle \cite{BrendleConvergenceInHihgerDimensions} showed that this is more generally the case for manifolds whose product with $\R^2$ and $\R$, respectively, has positive isotropic curvature. A crucial observation is that these curvature conditions, as well as the corresponding nonnegativity conditions, are preserved by the Ricci flow. 

In contrast, B\"ohm-Wilking \cite{BW2} remarked that $3$-positivity is not preserved by the Ricci flow in dimensions $n \geq 5.$ However, notice that for $n=5,6$ Theorem \ref{BettiNumberEstimate} implies a vanishing result for the Betti numbers of manifolds with $3$-positive curvature operator. 

In regard to Ricci flow invariant curvature conditions the following example is also noteworthy:

\begin{example*} \normalfont
Doubly warped product metrics on $S^n$ show that in dimensions $n \geq 6$ the class of manifolds which satisfy the curvature condition $\lambda_1 + \ldots + \lambda_{n-p} > 0$ of Theorem \ref{BettiNumberEstimate} is different from the class of manifolds with positive isotropic curvature (specifically for $p=1$ in dimensions $n \geq 6$ and for $p=\floor{\frac{n}{2}}$ in dimensions $n \geq 9)$. The two classes overlap but neither is contained in the other.

Furthermore, there are metrics on $S^n,$ $n \geq 5,$ which do not induce metrics of positive isotropic curvature on $S^n \times \R$, so that the eigenvalues of the curvature operator satisfy $\lambda_1 = \ldots = \lambda_p < 0$ and $\lambda_1 + \ldots + \lambda_{p+1} > 0$ for $p=2, \ldots, n-3.$

Example \ref{DoublyWarpedProducts} shows that these metrics can, in fact, be chosen $\mathcal{C}^1$-close to the round metric.
\end{example*}

Compact manifolds with positive isotropic curvature have been classified  by Hamilton \cite{HamiltonFourPIC}, Chen-Zhu \cite{ChenZhuRFwithSurgeryFourPIC} and Chen-Tang-Zhu \cite{ChenTangZhuClassFourPIC} in dimension $n=4$ and  by Brendle \cite{BrendleRFwithSurgeryPIC} and Huang \cite{HuangManifoldsPIC} in dimensions $n \geq 12$.


\vspace{2mm}

In view of Theorem \ref{BettiNumberEstimate} it is natural to ask: 

\begin{question*} \normalfont
Are there closed, simply connected Riemannian manifolds with $\lambda_1 + \ldots + \lambda_{n-1} >0$ and large second Betti number? Notice that $\C P^2$ is $3$-positive with $b_2 = 1.$ 

Are there closed, simply connected Riemannian manifolds with $\lambda_1 + \ldots + \lambda_{\ceil{\frac{n}{2}}} >0$ and torsion in homology? 

It is currently not known if manifolds with $\lambda_1 + \ldots + \lambda_{\ceil{\frac{n}{2}}} >0$ are diffeomorphic to space forms.
\end{question*}

In \cite{HoelzelSurgeryStableCurvCond} Hoelzel established a surgery procedure for manifolds that satisfy a point-wise curvature condition. For instance, this generalizes Micallef-Wang's \cite{MicallefWangNIC} result that positive isotropic curvature is preserved under connected sums. \vspace{2mm}

Many of the above mentioned results also have rigidity analogues in case of the corresponding nonnegativity conditions. In the context of the Bochner technique this goes back to Gallot-Meyer \cite{GallotMeyerCurvOperatorAndForms} who considered manifolds with nonnegative curvature operator. The more general results due to Ni-Wu \cite{NiWuNonnegativeCurvatureOperator}, Brendle-Schoen 
\cite{BrendleSchoenWeaklyQuarterPinched}, Seshadri \cite{SeshadriNIC} and Brendle \cite{BrendleEinsteinNIC} again rely on Ricci flow techniques. 

Cheeger \cite{CheegerVanishingPiecewiseConstCurv} adapted the Bochner technique to singular spaces and proved a vanishing theorem for spaces with positive piecewise constant curvature, as well as the corresponding rigidity theorem. As Cheeger  points out, these results indicate that spaces with nonnegative piecewise constant curvature may be regarded as a non-smooth analogue of manifolds with nonnegative curvature operator.
\vspace{2mm}

Based on work of P. Li \cite{LiSobolevConstant}, Gallot \cite{GallotSobolevEstimates} further generalized the Bochner technique and proved Theorem \ref{BettiNumberEstimate} in the case that the curvature operator is bounded from below by $\kappa \leq 0$ and the diameter is bounded above by $D>0.$ The proof of Theorem \ref{BettiNumberEstimate} also relies on the techniques developed by P. Li and Gallot. 

Remarkably, in the context of sectional curvature Gromov \cite{GromovCurvDiamBetti} established similar bounds on the Betti numbers using purely geometric ideas. 

\vspace{2mm}

With regard to classification results for manifolds with a nonnegativity condition on the sum of the lowest eigenvalues, Theorem \ref{BettiNumberEstimate} is mainly interesting in the case of generic holonomy. Otherwise it reduces to previous results due to Gallot-Meyer \cite{GallotMeyerCurvOperatorAndForms}, B\"ohm-Wilking \cite{BW2} and Mok \cite{MokUniformizationKaehler}:

\begin{remark*} \normalfont
Suppose that $(M,g)$ is $n$-dimensional and locally reducible. If the curvature operator is $(n-1)$-nonnegative, then the curvature operator is nonnegative. Similarly, if $\lambda_{n} >0,$ then $\lambda_1 = \ldots = \lambda_{n-1}= 0.$

Suppose that $(M,g)$ is  $n$-dimensional, locally irreducible, and has special holonomy. If the curvature operator is $\left( \frac{1}{4}n(n-2) \right)$-nonnegative, then the curvature operator is nonnegative. Similarly, if $\lambda_{\frac{1}{4}n(n-2)+1} >0,$ then $\lambda_1 = \ldots = \lambda_{\frac{1}{4}n(n-2)}= 0.$
\end{remark*}

Combined with Theorem \ref{BettiNumberEstimate} these observations lead to the following result.

\begin{corollary*}
Let $(M,g)$ be a closed connected $n$-dimensional Riemannian manifold with restricted holonomy $SO(n)$. 

If the eigenvalues of the curvature operator satisfy $\lambda_1 + \ldots + \lambda_{\ceil{\frac{n}{2}}} \geq 0$, then $b_p(M)=0$ for $0 < p < n.$
\end{corollary*}

Another application of our method yields a generalization of a theorem due to Tachibana \cite{TachibanaPosCurvOperator}. 

\begin{maintheorem}
Let $(M,g)$ be a closed connected $n$-dimensional Einstein manifold. If the eigenvalues $\lambda_1 \leq \ldots \leq \lambda_{\binom{n}{2}}$ of the curvature operator satisfy 
\begin{align*}
 \lambda_1 + \lambda_2 \geq 0  & \ \text{for} \ n=4 \ \text{or} \\
 \lambda_1 + \ldots + \lambda_{\floor{\frac{n-1}{2}}} \geq 0 &   \ \text{for} \ n \geq 5,
\end{align*}
then the curvature tensor is parallel. Moreover, if the inequality is strict, then $(M,g)$ has constant sectional curvature.
\label{GeneralizationTachibana}
\end{maintheorem}

In the case of $2$-nonnegative curvature operators this follows from the corresponding classification result due to Ni-Wu \cite{NiWuNonnegativeCurvatureOperator} and the fact that Einstein metrics are fixed points of the Ricci flow. Similarly, the rigidity results due to Brendle-Schoen 
\cite{BrendleSchoenWeaklyQuarterPinched} and Seshadri \cite{SeshadriNIC} yield Tachibana-type theorems. Brendle \cite{BrendleEinsteinNIC} specifically considers Einstein manifolds and shows that Einstein manifolds with nonnegative isotropic curvature are locally symmetric. In dimension $n=4$ this was observed by Micallef-Wang \cite{MicallefWangNIC}.

\vspace{2mm}

The proofs of Theorems \ref{BettiNumberEstimate} and \ref{GeneralizationTachibana} are based on a slight generalization of Poor's \cite{PoorHolonomyProofPosCurvOperatorThm} approach to the Hodge Laplacian. Poor's idea was to consider the derivative of the regular representation on tensors, and then to show that this leads to a simple formula for the curvature term in Lichnerowicz Laplacians. Lemma \ref{GeneralBochnerTermEstimate} offers a new method to control the curvature term based on an understanding how elements of $\mathfrak{so}(n)$ interact with tensors of a specific type. The work of P. Li \cite{LiSobolevConstant} and Gallot \cite{GallotSobolevEstimates} then implies a bound on the dimension of the kernel of the Lichnerowicz Laplacian, see theorem \ref{BochnerTechniqueMachinery}. 

Theorems \ref{BettiNumberEstimate} and \ref{GeneralizationTachibana} are applications of lemma \ref{GeneralBochnerTermEstimate} to $p$-forms and algebraic curvature tensors. Corollary \ref{BochnerForWeyl} covers the particular case of Weyl tensors. An application to $(0,2)$-tensors that instead uses averages of complex sectional curvatures is given in proposition \ref{BochnerZeroTwoTensors}. The required estimates to apply lemma \ref{GeneralBochnerTermEstimate} are established in lemma \ref{EstimatesForTensors} and proposition \ref{NormsOfHats}.

Section \ref{PreliminariesSection} reviews the relevant background material. The key technical lemmas are given in section \ref{TechnicalSection}. The proofs of the main theorems and other geometric applications follow in section \ref{GeometricApplicationsSection}. Section \ref{ExampleSection} contains details of the above doubly warped product metrics on $S^n$ and examples that show that the estimates in section \ref{TechnicalSection} are optimal. Furthermore, it exhibits an $(n-1)$-positive algebraic curvature operator and a $2$-form which yield a negative curvature term in the Bochner formula. Finally, it includes examples that can be used to give a different proof of parts of proposition \ref{NormsOfHats}.

 \vspace{2mm}

General references for background on the Bochner technique are B\'erard \cite{BerardBochnerRevisited}, Goldberg \cite{GoldbergCurvautreAndHolomogy} and Petersen \cite{PetersenRiemGeom}.

 \vspace{2mm}

\textit{Acknowledgments.} We would like to thank Christoph B{\"o}hm for constructive comments on a previous version of the paper.

\section{Preliminaries}
\label{PreliminariesSection}

\subsection{Tensors}
Let $(V,g)$ be an $n$-dimensional Euclidean vector space. The vector space of $(0,k)$-tensors on $V$ will be denoted by $\mathcal{T}^{(0,k)}(V)$ and the vector space of symmetric $(0,2)$-tensors by $\operatorname{Sym}^2(V).$

Recall that there is an orthogonal decomposition 
\begin{equation*}
\operatorname{Sym}^2(\Lambda^2V) = \operatorname{Sym}_B^2(\Lambda^2V) \oplus \Lambda^4V,
\end{equation*}
where the vector space $\operatorname{Sym}_B^2(\Lambda^2V)$ consists of all tensors $T \in \operatorname{Sym}^2(\Lambda^2V)$ that also satisfy the first Bianchi identity. Any $R \in \operatorname{Sym}_B^2(\Lambda^2V)$ is called an algebraic curvature tensor. 

The following norms and inner products for tensors, whose components are with respect to an arbitrary choice of an orthonormal basis, will be used throughout: When $T \in \mathcal{T}^{(0,k)}(V)$ define
\begin{equation*}
|T|^2 = \sum_{i_1, \ldots, i_k} \left( T_{i_1 \ldots i_k} \right)^2
\end{equation*}
whereas for a $p$-form $\omega \in \Lambda^p V^{*}$ set
\begin{equation*}
| \omega |^2 = \sum_{i_1 < \ldots < i_p} \left( \omega_{i_1 \ldots i_p} \right)^2.
\end{equation*}

Similarly, if $ \lbrace e_i \rbrace_{i=1, \ldots, n}$ is an orthonormal basis for $V$, then $\left\lbrace  e_{i_1} \wedge \ldots \wedge e_{i_p} \right\rbrace _{1 \leq i_1 < \ldots < i_p \leq n}$ is an orthonormal basis for $\Lambda^p V.$ This also induces an inner product on $\mathfrak{so}(V)$ via its identification with $\Lambda^2V.$ \vspace{2mm}

The Kulkarni-Nomizu product of $S, T \in \operatorname{Sym}^2(V)$ is given by
\begin{align*}
(S \owedge T)(X,Y,Z,W) = & \ S(X,Z)T(Y,W)-S(X,W)T(Y,Z) \\
& +S(Y,W)T(X,Z)-S(Y,Z)T(X,W).
\end{align*}
In particular, the tensor
\begin{equation*}
(g \owedge g)(X,Y,Z,W) = 2 \left\{ g(X,Z)g(Y,W)-g(X,W)g(Y,Z) \right\}
\end{equation*}
corresponds to the curvature tensor of the sphere of radius $1/\sqrt{2}.$

\begin{proposition}
If $h \in \operatorname{Sym}^2(V)$, then 
\begin{equation*}
| g \owedge h |^2 = 4(n-2) |h |^2 + 4 \tr(h)^2.
\end{equation*}
In particular, $| g \owedge g |^2 = 8 (n-1)n.$
\label{NormOfKNProduct}
\end{proposition}
\begin{proof}
By using an orthonormal basis $\left\lbrace e_i \right\rbrace$  for $V$ that diagonalizes $h$ one obtains:
\begin{equation*}
( g\owedge h)_{ijkl}  = 
\begin{cases} 
 h_{ii} + h_{jj} & \ \text{if} \ i=k \neq j=l, \\
-h_{ii} - h_{jj} & \ \text{if} \ i=l \neq j=k, \\
0 & \ \text{otherwise}.
\end{cases}
\end{equation*}
Hence
\begin{align*}
| g \owedge h |^2 & = \sum_{i,j,k,l} | (g \owedge h)_{ijkl} |^2 = 4 \sum_{i<j, k<l} | (g \owedge h)_{ijkl} |^2 = 4 \sum_{i<j} (h_{ii} + h_{jj})^2 \\
& = 2 \sum_{i \neq j} (h_{ii}+h_{jj})^2 = 2 \sum_{i, j} (h_{ii} + h_{jj})^2 -2 \sum_{i} (2 h_{ii})^2 \\ 
& = 2 ( 2n |h|^2 + 2 \tr(h)^2 - 4 |h|^2 ) \\
& = 4(n-2) |h|^2 + 4 \tr(h)^2
\end{align*}
as claimed.
\end{proof}

Recall that every algebraic $(0,4)$-curvature tensor $\Rm$ satisfies the orthogonal decomposition
\begin{equation*}
\Rm = \frac{\scal}{2(n-1)n} g \owedge g + \frac{1}{n-2} g \owedge \mathring{\Ric} + W,
\end{equation*}
where $\Riczero = \Ric - \frac{\scal}{n} g$ is the trace-free Ricci tensor and $W$ denotes the Weyl part. The associated algebraic curvature operator $\mathfrak{R} \colon \Lambda^2 V \to \Lambda^2 V$ is defined by 
\begin{align*}
g( \mathfrak{R}( x \wedge y), z \wedge w ) = \Rm(x,y,z,w).
\end{align*}
Note that the induced algebraic curvature tensor $R \in \operatorname{Sym}_B^2(\Lambda^2V)$ satisfies
\begin{align*}
| \Rm |^2 = 4 | R |^2.
\end{align*}


\subsection{The regular representation} 
\label{SubsectionTheRegularRepresentation}
The derivative of the regular representation of $O(n)$ on $(V,g)$ induces a derivation on tensors: If $T \in \mathcal{T}^{(0,k)}(V)$ and $L \in \mathfrak{so}(V),$ then
\begin{align*}
(LT)(X_1, \ldots, X_k) = - \sum_{i=1}^k T(X_1, \ldots, LX_i, \ldots, X_k).
\end{align*}
Notice that the metric $g$ satisfies $Lg = 0$ for all $L \in \mathfrak{so}(V)$ since
\begin{equation*}
(Lg)(X,Y) = - g(LX,Y) - g(X,LY) = - g(LX,Y) + g(LX, Y) =0.
\end{equation*}

\begin{proposition} When $\sigma \in S_k$ is a permutation and $T \in \mathcal{T}^{(0,k)}(V)$, then 
\begin{equation*}
(LT) \circ \sigma = L (T \circ \sigma )
\end{equation*}
for all $L \in \mathfrak{so}(V).$

In particular, for $S, T \in \operatorname{Sym}^2(V)$ the Kulkarni-Nomizu product satisfies
\begin{equation*}
L ( S \owedge T ) = (LS) \owedge T + S \owedge (LT)
\end{equation*}
for all $L \in \mathfrak{so}(V).$
\label{LeibnizForKNProduct}
\end{proposition}
\begin{proof}
This is a straightforward calculation:
\begin{align*}
( ( LT ) \circ \sigma )(X_1, \ldots, X_k) & = (LT) \left( X_{\sigma(1)}, \ldots, X_{\sigma(k)} \right) \\
& = - \sum_{j=1}^k T \left( X_{\sigma(1)}, \ldots, L X_{\sigma(j)}, \ldots, X_{\sigma(k)} \right) \\
& = - \sum_{j=1}^k (T \circ \sigma)(X_1, \ldots, L X_j, \ldots, X_k) \\
& = (L ( T \circ \sigma) )(X_1, \ldots, X_k).
\end{align*}
Hence the claim follows from the observation that 
\begin{equation*}
S \owedge T = (S \otimes T) \circ \tau_{23} - (S \otimes T) \circ \tau_{24} +(S \otimes T) \circ \tau_{14} -(S \otimes T) \circ \tau_{13},
\end{equation*}
where $\tau_{ij}$ denotes the transposition of the $i^{\text{th}}$ and $j^{\text{th}}$ entries.
\end{proof}

\begin{proposition}
If $h \in \operatorname{Sym}^2(V)$ and $L \in \mathfrak{so}(V),$ then $\tr(Lh) = 0$.
\end{proposition}
\begin{proof} 
Let $\lbrace e_i \rbrace$ be an orthonormal basis for $V$. It follows that 
\begin{align*}
\tr(Lh) 
& = \sum_{i=1}^n (Lh)(e_i, e_i) 
= -2 \sum_{i=1}^n h( L(e_i), e_i) \\
& = -2 \sum_{i,j=1}^n g(L e_i, e_j) h(e_j, e_i) 
 = -2 \langle h, g( L \cdot, \cdot ) \rangle
\end{align*}
and $\langle h, g( L \cdot, \cdot ) \rangle = - \langle h, g( L \cdot, \cdot ) \rangle = 0$ due to the symmetries of $L$, $g$ and $h.$
\end{proof}

The information on how all $L \in \mathfrak{so}(V)$ interact with a fixed $T \in \mathcal{T}^{(0,k)}$ can be encoded in a tensor $\hat{T}$ with values in $\Lambda^2V.$

\begin{definition}
For $T \in \mathcal{T}^{(0,k)}(V)$ define $\hat{T} \in \Lambda^2V \otimes \mathcal{T}^{(0,k)}(V)$ implicitly by
\begin{equation*}
g( L, \hat{T}(X_1, \ldots, X_k)) = (LT)(X_1, \ldots, X_k)
\end{equation*}
for all $L \in \mathfrak{so}(V) = \Lambda^2V$. 
\end{definition}

Notice that if $\lbrace \Xi_{\alpha} \rbrace$ is an orthonormal basis for $\mathfrak{so}(V) = \Lambda^2V$, then 
\begin{align*}
\hat{T} = \sum_{\alpha} \Xi_{\alpha} \otimes \Xi_{\alpha} T.
\end{align*}
Consequently, 
\begin{equation*}
| \hat{T} |^2 = \sum_{\alpha} | \Xi_{\alpha} T |^2.
\end{equation*}

\begin{exampleroman}
Let $e_1, \ldots, e_n$ be an orthonormal basis for $V$ with dual basis $e^1, \ldots, e^n$ and let $1 \leq i_1 < \ldots < i_p \leq n.$ It is simple to verify that
\begin{align*}
 \widehat{e^{i_1} \wedge \ldots \wedge e^{i_p} } = \sum_{\substack{ j=1, \ldots, p \\ k \notin \lbrace i_1, \ldots, i_p \rbrace}}
 (-1)^{j} \
 e_{\min \lbrace k, i_j \rbrace} \wedge e_{\max \lbrace k, i_j \rbrace } \
  e^k \wedge e^{i_1} \wedge \ldots \wedge \widehat{e^{i_j}} \wedge \ldots \wedge e^{i_p}.
\end{align*}
\end{exampleroman}

The following observation will be crucial for applications to the Bochner technique.

\begin{proposition}
Let $\mathfrak{R} \colon \Lambda^2 V \to \Lambda^2V$ be an algebraic curvature operator and $\lbrace \Xi_{\alpha} \rbrace$ an orthonormal basis for $\Lambda^2V$. It follows that 
\begin{align*}
\mathfrak{R}(\hat{T}) = \mathfrak{R} \circ \hat{T} = \sum_{\alpha} \mathfrak{R}( \Xi_{\alpha} ) \otimes \Xi_{\alpha} T.
\end{align*}
Furthermore, if $\lbrace \Xi_{\alpha} \rbrace$ is an eigenbasis of $\mathfrak{R}$ and $\lbrace \lambda_{\alpha} \rbrace$ denote the corresponding eigenvalues, then 

\begin{align*}
g( \mathfrak{R}(\hat{T}),\hat{T}) = \sum_{\alpha, \beta} g( \mathfrak{R}(\Xi_{\alpha}), \Xi_{\beta}) g(\Xi_{\alpha} T , \Xi_{\beta} T) = \sum_{\alpha} \lambda_{\alpha} | \Xi_{\alpha} T |^2. 
\end{align*} 
\label{ComputationBochnerTermWithEigenvalues}
\end{proposition}
The following formulae will be useful for the computation of examples:

\begin{proposition}
If $\lbrace \Xi_{\alpha} \rbrace$ is an orthonormal basis for $\Lambda^2V$ that diagonalizes $R \in \operatorname{Sym}^2(\Lambda^2V)$ and $\lbrace \lambda_{\alpha} \rbrace$ denote the corresponding eigenvalues, then 
\begin{equation*}
| LR |^2 = 2 \sum_{\alpha < \beta} ( \lambda_{\alpha} - \lambda_{\beta} )^2 g( L \Xi_{\alpha}, \Xi_{\beta} )^2
\end{equation*}
for every $L \in \mathfrak{so}(V).$
\label{NormLRinTermsOfEigenvalues}
\end{proposition}
\begin{proof} This is a straightforward calculation:
\begin{align*}
| LR |^2 & = \sum_{ \alpha, \beta } \left( (LR)( \Xi_{\alpha}, \Xi_{\beta} ) \right)^2 \\
& = \sum_{ \alpha, \beta } \left( -R( L\Xi_{\alpha}, \Xi_{\beta} ) - R( \Xi_{\alpha}, L \Xi_{\beta} ) \right)^2 \\
& = \sum_{\alpha, \beta}  \left( - \lambda_{\beta} g( L \Xi_{\alpha}, \Xi_{\beta} ) - \lambda_{\alpha} g( \Xi_{\alpha}, L \Xi_{\beta} ) \right)^2 \\
& = \sum_{\alpha, \beta} ( \lambda_{\alpha} - \lambda_{\beta} )^2 g( L \Xi_{\alpha}, \Xi_{\beta} )^2.
\end{align*}
\end{proof}

\begin{proposition} If $\lbrace e_i \rbrace$ is an orthonormal basis for $V$ that diagonalizes $h \in \operatorname{Sym}^2(V)$ and $\lbrace h_i \rbrace$ denote the corresponding eigenvalues, then 
\begin{align*}
| L h|^2 = 2 \sum_{i < j} (h_i - h_j)^2 g( L(e_i), e_j )^2 \leq 2 (h_{\max} - h_{\min})^2 |L|^2
\end{align*}
for all $L \in \mathfrak{so}(V)$. It follows that
\begin{align*}
| \hat{h} |^2 = 2n |h|^2 - 2 \tr(h)^2 = 2n | \mathring{h} |^2 .
\end{align*}
\label{NormActionOnSymTensors}
\end{proposition}
\begin{proof} $| L h|^2$ is calculated as in proposition \ref{NormLRinTermsOfEigenvalues} and 
\begin{align*}
| \hat{h} |^2 
& = \sum_{k<l} | (e_k \wedge e_l) h |^2 \\
& = \sum_{k<l} \sum_{i,j} (h_i - h_j)^2 g( (e_k \wedge e_l) e_i, e_j)^2 \\
& = \sum_{k<l} \sum_{i,j} (h_i - h_j)^2 g(\delta_{ki} e_l - \delta_{li} e_k, e_j)^2 \\
& = \sum_{k<l} \sum_{i,j} (h_i - h_j)^2 ( \delta_{ki}\delta_{lj} - \delta_{li}\delta_{kj})^2 \\
& = \sum_{k < l} (h_k - h_l)^2 +  \sum_{k < l} (h_l - h_k)^2 \\
& = \sum_{k,l} (h_k - h_l)^2 \\
& = 2n |h|^2 - 2 \tr(h)^2
\end{align*}
as claimed.
\end{proof}

\subsection{The Bochner Technique}
\label{BochnerTechniqueSection}
Let $(M,g)$ be a closed $n$-dimensional Riemannian manifold and let $R(X,Y)Z = \nabla_Y \nabla_X Z - \nabla_X \nabla_Y Z + \nabla_{[X,Y]} Z$ denote its curvature tensor. For $T \in \mathcal{T}^{(0,k)}(M)$ set 
\begin{equation*}
\Ric(T)(X_1, \ldots, X_k) = \sum_{i=1}^k \sum_{j=1}^n  (R(X_i,e_j)T) (X_1, \ldots, e_j, \ldots, X_k).
\end{equation*}

\begin{remarkroman}
Recall that the Ricci identity asserts
\begin{align*}
R(X,Y) T(X_1, \ldots, X_k) = - \sum_{i=1}^k T(X_1, \ldots, R(X,Y) X_i, \ldots, X_k),
\end{align*}
which is in agreement with the effect of $R(X,Y) \in \mathfrak{so}(TM)$ on $T \in \mathcal{T}^{(0,k)}(M)$ defined in section \ref{SubsectionTheRegularRepresentation}. In particular the above definition of $\Ric(T)$ carries over to algebraic curvature tensors. The notation $\Ric_R(T)$ will be used to specify the algebraic curvature tensor $R.$ 
\end{remarkroman}

Let $E \to M$ be a subbundle of $\mathcal{T}^{(0,k)}(M)$. For $c>0$ the {\em Lichnerowicz Laplacian} on $E$ is given by 
\begin{equation*}
\Delta_L = \nabla^{*} \nabla + c \Ric.
\end{equation*}
A tensor $T$ is called {\em harmonic} if $\Delta_L  T =0.$

\begin{exampleroman}
There are various important examples of Lichnerowicz Laplacians for different $c>0.$
\begin{enumerate}
\item The Hodge Laplacian is a Lichnerowicz Laplacian for $c=1$ and a $p$-form $\omega$ is harmonic if and only if it is closed and divergence free.
\item The natural definition of the Lichnerowicz Laplacian for symmetric $(0,2)$-tensors uses $c=\frac{1}{2}$. With this choice $h \in \operatorname{Sym}^2(M)$ is harmonic if and only if $h$ is a Codazzi tensor and divergence free. This is equivalent to $h$ being Codazzi and having constant trace. This has been used by Berger \cite{BergerCompactEinstein}, \cite{BergerKaehlerEinstein} in the case of Einstein metrics and by Simons \cite{SimonsMinimalVarieties} in the case of constant mean curvature hypersurfaces.
\item The Lichnerowicz Laplacian for algebraic curvature tensors $\Rm$ on a Riemannian manifold also uses $c=\frac{1}{2}.$ With this choice $\Rm$ is harmonic if it satisfies the second Bianchi identity and it is divergence free. If $\Rm$ satisfies the second Bianchi identity, then it is divergence free if and only if its Ricci tensor is a Codazzi tensor, and in this case its scalar curvature is constant. This was used by Tachibana \cite{TachibanaPosCurvOperator}. 
\end{enumerate}
\label{ExamplesLichnerowiczLaplacians}
\end{exampleroman}

The next proposition is established in \cite[lemmas 9.3.3 and 9.4.3]{PetersenRiemGeom}.
\begin{proposition}
If $S, T \in \mathcal{T}^{(0,k)}(M)$, then
\begin{equation*}
g( \Ric(S), T) = g( \mathfrak{R}( \hat{S}), \hat{T} ). 
\end{equation*}
In particular, $\Ric$ is self-adjoint.
\end{proposition}

The following theorem summarizes the framework of the Bochner technique for general Lichnerowicz Laplacians. In this form it is due to the work of P. Li \cite{LiSobolevConstant} and Gallot \cite{GallotSobolevEstimates}.

\begin{theorem}
Let $n \geq 3,$ $\kappa \leq 0$ and $D>0$, and let $(M,g)$ be a closed connected $n$-dimensional Riemannian manifold with $\Ric(M) \geq (n-1) \kappa$ and $\diam(M) \leq D.$ 

Let $E \to M$ be a subbundle of $\mathcal{T}^{(0,k)}(M)$ with $m$-dimensional fiber and assume there is $C>0$ such that 
\begin{equation*}
g( \mathfrak{R}( \hat{T}), \hat{T}  ) \geq \kappa C | T |^2
\end{equation*}  for all $T \in \Gamma(E)$. 

In this case the dimension of the kernel of the associated Lichnerowicz Laplacian 
\begin{equation*}
\ker( \Delta_L ) = \left\lbrace  T \in \Gamma(E) \ \vert \ \Delta_L T = \nabla^{*} \nabla T + c \Ric(T) =  0 \right\rbrace 
\end{equation*}
is bounded by
\begin{equation*}
m \cdot \exp \left( C \left( n, \kappa D^2 \right) \cdot \sqrt{-\kappa D^2 c C} \right)
\end{equation*}
and when $\kappa =0$, then all $T \in \ker( \Delta_L )$ are parallel. 

Moreover, there is $\varepsilon(n,cC) >0$ such that $\kappa D^2 \geq - \varepsilon(n,cC)$ implies $\dim \ker( \Delta_L ) \leq m$. Finally, if $g( \mathfrak{R}( \hat{T}), \hat{T} ) >0$ for all $T \in \Gamma(E)$ with $\hat{T} \neq 0$, then 
\begin{equation*}
\ker( \Delta_L ) = \lbrace T \in \Gamma(E) \ \vert \ T \ \text{parallel}, \ \hat{T} = 0 \rbrace.
\end{equation*}
\label{BochnerTechniqueMachinery}
\end{theorem}

\begin{remarkroman}
The condition $\Ric(M) \geq (n-1) \kappa$ is always satisfied in the situation of Theorem \ref{BettiNumberEstimate} and Theorem \ref{GeneralizationTachibana} since the Ricci curvature is bounded from below by the sum of the lowest $(n-1)$ eigenvalues of the curvature operator.
\end{remarkroman}

\section{Controlling the curvature term of Lichnerowicz Laplacians}
\label{TechnicalSection}

The following lemma  provides a general method of controlling the curvature term of the Lichnerowicz Laplacian on tensors.

\begin{lemma}
Let $\mathfrak{R} \colon \Lambda^2V \to \Lambda^2 V$ be an algebraic curvature operator with eigenvalues $\lambda_1 \leq \ldots \leq \lambda_{\binom{n}{2}}$ and let $T \in \mathcal{T}^{(0,k)}(V)$.

Suppose there is $C \geq 1$ such that 
\begin{equation*}
| L T |^2 \leq \frac{1}{C}| \hat{T}|^2 | L |^2 
\end{equation*}
for all $L \in  \mathfrak{so}(V).$

Let $\kappa \leq 0.$ If $\frac{1}{\floor{C}} \left( \lambda_1 + \ldots + \lambda_{\floor{C}} \right) \geq \kappa$, then $g( \mathfrak{R}( \hat{T} ), \hat{T} ) \geq \kappa |\hat{T}|^2$ and if $\lambda_1 + \ldots + \lambda_{\floor{C}} > 0$, then $g( \mathfrak{R}( \hat{T} ), \hat{T} ) > 0$ unless $\hat{T}=0.$
 \label{GeneralBochnerTermEstimate}
\end{lemma}
\begin{proof}
Choose an orthonormal basis $\left\lbrace \Xi_{\alpha} \right\rbrace$ for $\Lambda^2V$ such that $\mathfrak{R}( \Xi_{\alpha} ) = \lambda_{\alpha} \Xi_{\alpha}.$ Notice that $\lambda_{\floor{C}+1} \geq \kappa$, which in turn implies
\begin{align*}
g( \mathfrak{R}(\hat{T}),\hat{T}) 
& = \sum_{\alpha=1}^{\binom{n}{2}} \lambda_{\alpha}  |\Xi_{\alpha} T|^2 \\
& = \sum_{\alpha = \floor{C}+1}^{\binom{n}{2}} \lambda_{\alpha} |\Xi_{\alpha} T|^2  + \sum_{\alpha = 1}^{\floor{C}} \lambda_{\alpha} |\Xi_{\alpha} T|^2  \\
& \geq \lambda_{\floor{C}+1} \sum_{\alpha = \floor{C}+1}^{\binom{n}{2}} |\Xi_{\alpha} T|^2  + \sum_{\alpha = 1}^{\floor{C}} \lambda_{\alpha} |\Xi_{\alpha} T|^2  \\
& = \lambda_{\floor{C}+1} |\hat{T}|^2 + \sum_{\alpha = 1}^{\floor{C}} \left(  \lambda_{\alpha} - \lambda_{\floor{C}+1} \right) |\Xi_{\alpha} T|^2  \\
& \geq \lambda_{\floor{C}+1} | \hat{T}|^2 + \frac{1}{C} \sum_{\alpha = 1}^{\floor{C}} \left(  \lambda_{\alpha} - \lambda_{\floor{C}+1} \right) | \hat{T}|^2 \\
& =  \lambda_{\floor{C}+1} \left(1-\frac{\floor{C}}{C} \right) | \hat{T}|^2 + \frac{| \hat{ T}|^2}{C} \sum_{\alpha =1}^{\floor{C}} \lambda_{\alpha} \\
& \geq  \kappa | \hat{ T}|^2.
\end{align*}

The last claim follows from the observation that for $\lambda_{ \floor{C} + 1} \geq 0$ the above calculation implies $g( \mathfrak{R}(\hat{T}),\hat{T}) \geq  \frac{| \hat{ T}|^2}{C}  \sum_{\alpha =1}^{\floor{C}} \lambda_{\alpha}.$
\end{proof}

In the following, $|LT|^2$ will be estimated or computed for various types of tensors:

\begin{lemma}
Let $(V,g)$ be an $n$-dimensional Euclidean vector space and $L \in \mathfrak{so}(V).$ The following hold:
\begin{enumerate}
\item Every $T \in \mathcal{T}^{(0,k)}(V)$ satisfies
\begin{align*}
| L T |^2 \leq k^2 |T|^2 | L |^2.
\end{align*}
\item Every $h \in \operatorname{Sym}^2(V)$ satisfies
\begin{align*}
| L h |^2 \leq 4 | \mathring{h}|^2 | L |^2.
\end{align*}
\item Every $p$-form $\omega$ satisfies 
\begin{align*}
| L \omega |^2 \leq \min\lbrace p,n-p\rbrace | \omega|^2 |L|^2.
\end{align*}
\item Every $R \in \operatorname{Sym}^2( \Lambda^2V)$ satisfies
\begin{align*}
| LR |^2 \leq 8 | \mathring{R} |^2 |L|^2
\end{align*}
and the associated $(0,4)$-tensor $\Rm$ also satisfies 
\begin{align*}
| L \Rm |^2 \leq 8 | \mathring{\Rm}|^2 |L|^2.
\end{align*}
\end{enumerate}
\label{EstimatesForTensors}
\end{lemma}
\begin{proof}
Choose an orthonormal basis $\lbrace e_i \rbrace$ for $V$ so that 
\begin{equation*}
L = \sum_{i=1}^{\floor{n/2}} \alpha_{2i-\frac{1}{2}} e_{2i-1} \wedge e_{2i}
\end{equation*}
and observe that $L e_i = (-1)^{i+1} \alpha_{i+\frac{(-1)^{i+1}}{2}} e_{i+(-1)^{i+1}}$.

In case (a) this yields
\begin{align*}
T \left( e_{i_1}, \ldots, L e_{i_j}, \ldots, e_{i_k} \right) 
& = (-1)^{i_j+1} \alpha_{i_j+\frac{(-1)^{i_j+1}}{2}} \ T \left( e_{i_1}, \ldots,  e_{i_j+(-1)^{i_j+1}},  \ldots, e_{i_k} \right)
\end{align*}
and 
\begin{align*}
| (L T)(e_{i_1}, \ldots, e_{i_k} ) |^2 & =\left| -  \sum_{j=1}^k T \left( e_{i_1}, \ldots, L e_{i_j}, \ldots, e_{i_k} \right) \right|^2 \\
& =  \left| - \sum_{j=1}^k (-1)^{i_j+1} \alpha_{i_j+\frac{(-1)^{i_j+1}}{2}} \ T_{{i_1} \ldots  i_j+(-1)^{i_j+1}  \ldots {i_k}} \right|^2 \\
& \leq \left( \sum_{j=1}^k \left(\alpha_{i_j+\frac{(-1)^{i_j+1}}{2}} \right)^2 \right) \left( \sum_{j=1}^k \left(T_{{i_1} \ldots  i_j+(-1)^{i_j+1}  \ldots {i_k}} \right)^2 \right) \\
& \leq k | L |^2 \sum_{j=1}^k \left(T_{{i_1} \ldots  i_j+(-1)^{i_j+1}  \ldots {i_k}} \right)^2.
\end{align*}
Summation over $i_1, \ldots, i_k$ implies
\begin{align*}
| L T |^2 
\leq k |L|^2 \sum_{i_1, \ldots, i_k} \sum_{j=1}^k \left(T_{{i_1} \ldots  i_j+(-1)^{i_j+1}  \ldots {i_k}} \right)^2 
\leq k^2 |L|^2 |T|^2.
\end{align*}

Case (b) follows from (a) and the observation that the trace-free part $\mathring{h} = h - \frac{\tr(h)}{n}g$ satisfies $Lh = L \mathring{ h}$ for all $L \in \mathfrak{so}(V).$

It suffices to prove (c) for $p \leq \floor{\frac{n}{2}}$ due to Hodge duality. Furthermore, assume $i_1 < \ldots < i_p$ in the above calculation. It follows that the coefficients $\alpha_{i_j+\frac{(-1)^{i_j+1}}{2}}$ that are summed over all correspond to different coefficients of $L.$ Indeed, a coefficient can only occur twice if there are consecutive indices $k, l$ such that
\begin{align*}
i_{k}+\frac{1}{2} = i_{l}-\frac{1}{2}.
\end{align*}
However, in this case observe that
\begin{equation*}
\alpha_{i_k + \frac{1}{2}} \omega \left( e_{i_1}, \ldots, e_{i_{k}+1}, e_{i_{l}}, \ldots, e_{i_p} \right) = 0 \ \text{ and } \ \alpha_{i_l - \frac{1}{2}} \omega \left( e_{i_1}, \ldots, e_{i_{k}}, e_{i_{l}-1}, \ldots, e_{i_p} \right) = 0
\end{equation*}
and hence these terms do not occur in the summation. Thus
\begin{align*}
| (L \omega) \left(e_{i_1}, \ldots, e_{i_p} \right) |^2
 \leq | L |^2 \sum_{j=1}^p \left( \omega \left( e_{i_1}, \ldots,  e_{i_j+(-1)^{i_j+1}},  \ldots, e_{i_p} \right) \right)^2
\end{align*}
and summation over $i_1 < \ldots < i_p$ yields the claim. 

Case (d) follows as in (c) by using the symmetries of $\Rm.$
\end{proof}

\begin{remarkroman}
The examples in section \ref{ExampleSection} show that the estimates in lemma \ref{EstimatesForTensors} cannot be improved without further assumptions.
\end{remarkroman}

\begin{corollary}
Every $h \in \operatorname{Sym}^2(V)$ satisfies
\begin{equation*}
|L (g \owedge h) |^2 \leq 4  |g \owedge \mathring{h} |^2 |L|^2
\end{equation*}
for all $L \in \mathfrak{so}(V).$ 

Furthermore, every algebraic curvature tensor $\Rm$ satisfies
\begin{align*}
\left| L \Rm \right|^2 \leq \left( 4 \left| \frac{1}{n-2} g \owedge \Riczero \right|^2 + 8 \left| W \right|^2 \right) \left| L \right|^2
\end{align*}
for all $L \in \mathfrak{so}(V).$ 
\label{LRmWithVanishingWeyl}
\end{corollary}
\begin{proof} 
Lemma \ref{EstimatesForTensors} and the computations in section \ref{PreliminariesSection} immediately imply
\begin{align*}
|L (g \owedge h) |^2 
&= | g \owedge Lh |^2 = 4(n-2) |Lh|^2 + 4 \tr(Lh)^2 \\
& = 4(n-2) | L h |^2 \leq 16(n-2) | \mathring{h} |^2 |L|^2 \\
& = 4 |g \owedge \mathring{h} |^2 |L|^2.
\end{align*} 

Since $L \Rm$ respects the orthogonal decomposition of algebraic curvature tensors, it follows that
\begin{align*}
\left| L \Rm \right|^2 
=  | L \Rmzero |^2
= \left| \frac{1}{n-2} L (g \owedge \mathring{ \Ric}) + L W \right|^2 
=  \left| \frac{1}{n-2}  L (g \owedge \mathring{ \Ric}) \right|^2 + \left| L W \right|^2.
\end{align*}
Hence, the first part of this corollary and lemma \ref{EstimatesForTensors} imply the claim.
\end{proof}

In case $\dim V=4$ and $R$ is Einstein, $LR$ can be computed explicitly with the help of a Singer-Thorpe basis:

\begin{remarkroman}
Let $(V,g)$ be a $4$-dimensional Euclidean vector space. Due to results of Singer-Thorpe \cite{SingerThorpeEinstein}, an algebraic curvature operator $\mathfrak{R} \colon \Lambda^2 V \to \Lambda^2 V$ is Einstein if and only if it commutes with the Hodge star operator with respect to any orientation of $V$. Once an orientation is fixed, there is an orthonormal basis $e_1, e_2, e_3, e_4$ of $V$ such that 
\begin{align*}
\Xi_1 = \frac{1}{\sqrt{2}} \left( e_1 \wedge e_2 + e_3 \wedge e_4 \right), \ \
\Xi_2  = \frac{1}{\sqrt{2}} \left( e_1 \wedge e_3 - e_2 \wedge e_4 \right), \ \
\Xi_3  = \frac{1}{\sqrt{2}} \left(  e_1 \wedge e_4 + e_2 \wedge e_3 \right)
\end{align*}
is an orthonormal basis for the self-dual part $\Lambda^{+}V$,
\begin{align*}
\Xi_4 = \frac{1}{\sqrt{2}} \left( e_1 \wedge e_2 - e_3 \wedge e_4 \right), \ \
\Xi_5 = \frac{1}{\sqrt{2}} \left( e_1 \wedge e_3 + e_2 \wedge e_4 \right), \ \
\Xi_6 = \frac{1}{\sqrt{2}} \left(  e_1 \wedge e_4 - e_2 \wedge e_3 \right)
\end{align*}
is an orthonormal basis for the anti-self-dual part $\Lambda^{-}V$ and $\mathfrak{R}( \Xi_i ) = \lambda_i \Xi_i$ for $i =1 , \ldots, 6.$ Any such basis is called a Singer-Thorpe basis.

Moreover, with respect to the above basis $\lbrace \Xi_{\alpha} \rbrace$, the self-adjoint operator $\mathfrak{R} \colon \Lambda^2 V \to \Lambda^2V$ defined by $\mathfrak{R}( \Xi_i ) = \lambda_i \Xi_i$ for $i=1, \ldots, 6$ satisfies the first Bianchi identity if and only if $\lambda_1 + \lambda_2 + \lambda_3 = \lambda_4 + \lambda_5 + \lambda_6.$

Finally, notice that
\begin{equation*}
( \Xi_i ) \Xi_j = \begin{cases}
 \pm \sqrt{2} \Xi_k & \ \text{if } \ \lbrace i, j, k\rbrace= \lbrace 1,2,3 \rbrace \ \text{ or } \ \lbrace 4, 5,6 \rbrace \\
 0 & \ \text{otherwise.}
\end{cases}
\end{equation*}
\label{SingerThorpeBasis}
\end{remarkroman}

\begin{proposition}
Let $(V,g)$ be a $4$-dimensional oriented Euclidean vector space and suppose that $R \in \operatorname{Sym}_B^2(\Lambda^2V)$ is Einstein. Denote by $R_{\pm} \in \operatorname{Sym}_B^2(\Lambda^{\pm} V)$ the induced curvature tensors on $\Lambda^{\pm} V$, and by $L_{\pm}$ the orthogonal projections of $L \in \Lambda^2 V$ onto $\Lambda^{\pm}V$. It follows that
\begin{align*}
|LR|^2 = | L_{+} R_{+} |^2 + | L_{-} R_{-} |^2.
\end{align*}

In particular, if $\lbrace \Xi_{\alpha} \rbrace$ is a Singer-Thorpe basis for $\Lambda^2 V,$ $\lbrace \lambda_{\alpha} \rbrace$ denote the eigenvalues of $\mathfrak{R}$, and $L = \sum_{\alpha=1}^6 a_{\alpha} \Xi_{\alpha}$, then
\begin{equation*}
| LR |^2 = 4 \sum_{\gamma=1}^6 a_{\gamma}^2 ( \lambda_{\alpha} - \lambda_{\beta} )^2 \leq 4 \left( \lambda_{\max} - \lambda_{\min} \right)^2 | L |^2,
\end{equation*}
where the summation indices are such that $\alpha < \beta$ and $\Xi_{\alpha}, \Xi_{\beta}, \Xi_{\gamma}$ form a basis for $\Lambda^{+} V$ or $\Lambda^{-}V.$
\label{LRinSingerThorpeBasis}
\end{proposition}
\begin{proof}
Notice that a Singer-Thorpe basis $\lbrace \Xi_{\alpha} \rbrace$ satisfies $g( L_{+} \Xi_{\alpha}, \Xi_{\beta} )=0$ if $\Xi_{\alpha} \in \Lambda^{-}V$ or $\Xi_{\beta} \in \Lambda^{-}V.$ Thus proposition \ref{NormLRinTermsOfEigenvalues} implies 
\begin{align*}
|L_{+} R |^2 
& = \sum_{\alpha, \beta = 1, \ldots, 6} ( \lambda_{\alpha} - \lambda_{\beta} )^2 g( L_{+} \Xi_{\alpha}, \Xi_{\beta} )^2 \\
& = \sum_{\alpha, \beta = 1, 2, 3 } ( \lambda_{\alpha} - \lambda_{\beta} )^2 g( L_{+} \Xi_{\alpha}, \Xi_{\beta} )^2 \\
&   = | L_{+} R_{+} |^2.
\end{align*}
Similarly one proves $ |L_{-} R |^2 = | L_{-} R_{-} |^2$ and the formulae for $|LR|^2.$ 
\end{proof}

\begin{remarkroman}
The curvature operator of $\C P^2$ with the Fubini-Study metric provides an example with equality in the estimate above: There is a Singer-Thorpe basis for the curvature operator of $(\C P^2,g_{FS})$ such that $\mathfrak{R}( \Xi_1 ) = \mathfrak{R}( \Xi_2 ) = 0,$ $\mathfrak{R}( \Xi_3 ) = 6 \Xi_3$ and $\mathfrak{R}( \Xi_i ) = 2 \Xi_i$ for $i=4,5,6.$ Thus $| \Xi_1 R|$ and $| \Xi_2 R|$ are maximal.
\label{LREstimateOptimalForCPTwo}
\end{remarkroman}

Let $\id \wedge \id$ denote the curvature tensor of the unit sphere. The computation of $| \hat{T} |^2$ in the propositions below relies on the observation that $|\hat{T}|^2 = g( \Ric_{\id \wedge \id}(T), T)$. 

\begin{proposition}
Let $(V,g)$ be an $n$-dimensional Euclidean vector space and  $T \in \mathcal{T}^{(0,k)}(V)$. It follows that
\begin{align*}
\Ric_{\id \wedge \id}(T) (X_1, \ldots, X_k) 
= & \ k (n-1)T (X_1, \ldots, X_k) + \sum_{i \neq j} ( T \circ \tau_{ij}) (X_1, \ldots, X_k) \\
  & \ - \sum_{i \neq j} g(X_i, X_j) c_{ij}(T)( X_1, \ldots, \widehat{X}_i, \ldots, \widehat{X}_j, \ldots, X_k ),
\end{align*}
where $\tau_{ij}$ denotes the transposition of the $i^{\text{th}}$ and $j^{\text{th}}$ entries and $c_{ij}(T)$ is the contraction of $T$ in the $i^{\text{th}}$ and $j^{\text{th}}$ entries.
\label{GeneralFormulaBochnerOnIdentity}
\end{proposition}
\begin{proof}
Recall that the curvature tensor of the unit sphere satisfies 
\begin{equation*}
R(X,Y)Z = g(X,Z)Y - g(Y,Z)X = (X \wedge Y)(Z).
\end{equation*}
Let $\lbrace e_i \rbrace$ be an orthonormal basis for $V$. The claim now follows from the calculation:
\begin{align*}
\Ric_{\id \wedge \id}(T)(X_1, \ldots, X_k) 
= & \ \sum_{i=1}^k \sum_{a=1}^n ( R(X_i,e_a)T)(X_1, \ldots, e_a, \ldots, X_k) \\
= & \ \sum_{i \neq j} \sum_{a=1}^n T( X_1, \ldots, (e_a \wedge X_i) X_j, \ldots, e_a, \ldots, X_k) \\
& \ + \sum_{i=1}^k \sum_{a=1}^n T(X_1, \ldots, (e_a \wedge X_i)e_a, \ldots, X_k) \\
= & \ \sum_{i \neq j} \sum_{a=1}^n T( X_1, \ldots, g(e_a, X_j) X_i - g(X_i, X_j)e_a, \ldots, e_a, \ldots, X_k) \\
& \ + \sum_{i=1}^k \sum_{a=1}^n T(X_1, \ldots, X_i- g(e_a, X_i)e_a, \ldots, X_k) \\
= & \sum_{i \neq j} \sum_{a=1}^n T( X_1, \ldots, X_i, \ldots, g(e_a, X_j)e_a, \ldots, X_k) \\
& \ - \sum_{i \neq j} \sum_{a=1}^n g(X_i, X_j) T(X_1, \ldots, e_a, \ldots, e_a, \ldots, X_k) \\
& \ + k(n-1) T(X_1, \ldots, X_k).
\end{align*}
\end{proof}

\begin{proposition}
Let $(V,g)$ be an $n$-dimensional Euclidean vector space and let $\id \wedge \id$ denote the curvature tensor of the unit sphere. The following hold:
\begin{enumerate}
\item Every $h \in \operatorname{Sym}^2(V)$ satisfies
\begin{align*}
\Ric_{\id \wedge \id}(h) & = 2n \mathring{h}, \\
| \hat{h} |^2 & = | \hat{\mathring{h}} |^2 = 2n | \mathring{h} |^2.
\end{align*}
\item Every $p$-form $\omega$ satisfies
\begin{align*}
\Ric_{\id \wedge \id}(\omega) & = p(n-p) \omega, \\
| \hat{\omega} |^2 & = p(n-p) | \omega |^2.
\end{align*}
\item Every algebraic $(0,4)$-curvature tensor $\Rm$ and every $R \in \operatorname{Sym}_B^2(\Lambda^2V)$ satisfy
\begin{align*}
\Ric_{\id \wedge \id}(\Rm) & = 4(n-1) \Rm - 2 g \owedge \Ric, \\
| \widehat{\Rm} |^2 & = | \widehat{\mathring{\Rm}} |^2 = 4(n-1) | \Rmzero |^2 - 8 | \Riczero |^2, \\
| \hat{R} |^2 & = | \hat{\mathring{R}} |^2  = 4 (n-1) | \mathring{R} |^2  - 2 | \Riczero |^2 .
\end{align*}
In particular $\widehat{\Rm} = 0$ if and only if $\Rm = \frac{\kappa}{2} g \owedge g$ for some $\kappa \in \R.$
\end{enumerate}

\label{NormsOfHats}
\end{proposition}
\begin{proof}
(a) Notice that $h \circ \tau_{ij} = h$ for every transposition $\tau_{ij}$. Hence
\begin{align*}
\Ric_{\id \wedge \id}(h) = h \circ \tau_{12} + h \circ \tau_{21} - 2 \tr(h) g + 2(n-1) h = 2n \left( h - \frac{\tr(h)}{n}g \right) = 2n  \mathring{h}.
\end{align*}
(b) Similarly $\omega \circ \tau_{ij} = - \omega$ for every transposition $\tau_{ij}$ and thus $c_{ij}(T) = 0$ for all $i \neq j$. This implies
\begin{align*}
\Ric_{\id \wedge \id}(\omega) = - \sum_{i \neq j} \omega + p(n-1) \omega = p(n-p) \omega.
\end{align*}
(c) Due to the symmetries of the curvature tensor
\begin{align*}
 \sum_{i \neq j} \Rm \circ \tau_{ij} & = 2 ( \Rm \circ \tau_{12} +  \Rm \circ \tau_{13} + \Rm \circ \tau_{14} +  \Rm \circ \tau_{23} +  \Rm \circ \tau_{24}+  \Rm \circ \tau_{34} ) \\
 & = -4 \Rm + 2 (  \Rm \circ \tau_{13} +  \Rm \circ \tau_{14} + \Rm \circ \tau_{23} + \Rm \circ \tau_{24} )
\end{align*}
which implies
\begin{align*}
 \sum_{i \neq j} \left( \Rm \circ \tau_{ij} \right) (X,Y,Z,W) 
 = & \ -4 \Rm(X,Y,Z,W) + 2 \lbrace \Rm( Z, Y, X, W) + \Rm( W, Y, Z, X)  \\
& \ +  \Rm( X, Z, Y, W)+ \Rm( X, W, Z, Y) \rbrace \\
= & \ -4 \lbrace \Rm(X,Y,Z,W) + \Rm(Y,Z,X,W) + \Rm(Z,X,Y,W)  \rbrace \\
= & \ 0
\end{align*}
due to the first Bianchi identity. For the remaining term one computes
\begin{align*}
\sum_{i \neq j} ( g( \cdot, \cdot) c_{ij}(\Rm)) & (X, Y, Z, W) \\
= & \ 2 \sum_{i=1}^n \lbrace g(X,Z) \Rm(e_i, Y, e_i, W) + g(X,W) \Rm(e_i, Y, Z, e_i) \\
  & \ + g(Y,Z) \Rm(X, e_i, e_i, W ) + g(Y,W) \Rm(X, e_i, Z, e_i) \rbrace \\
= & \ 2 \sum_{i=1}^n \lbrace g(X,Z) \Rm(Y, e_i, W, e_i) - g(X,W) \Rm(Y, e_i, Z,  e_i) \\
  & \ - g(Y,Z) \Rm( X, e_i, W,  e_i) + g(Y,W) \Rm(X, e_i, Z, e_i) \rbrace  \\
= & 2 (g \owedge \Ric)(X, Y, Z,W).
\end{align*}
To calculate $| \widehat{\Rm} |^2$ observe that 
\begin{equation*}
g( \Rm, g \owedge \Ric ) = \frac{\scal^2}{2n^2(n-1)} | g \owedge g|^2 + \frac{1}{n-2} | \mathring{\Ric} \owedge g |^2 = 4 \frac{\scal^2}{n} + 4 | \Riczero |^2
\end{equation*}
due to proposition \ref{NormOfKNProduct}. For the last claim observe that 
\begin{align*}
| \widehat{\Rm} |^2 & = 4 (n-1) | \Rmzero |^2 - 8 | \Riczero |^2 \\
& = 4(n-1) \left(\frac{1}{(n-2)^2} | g \owedge \Riczero |^2 + | W |^2 \right) - 8 | \Riczero |^2 \\
& = \left( 16 \frac{n-1}{n-2} - 8 \right) | \Riczero |^2 + 4 (n-1) |W|^2.
\end{align*}
In particular, $| \widehat{\Rm} |^2 = 0$ is equivalent to $| \Rmzero |^2=0$.
\end{proof}

\section{Geometric Applications}
\label{GeometricApplicationsSection}

This section contains the proofs of Theorem  \ref{BettiNumberEstimate} and Theorem \ref{GeneralizationTachibana}, and provides further applications of the Bochner technique to Weyl tensors and $(0,2)$-tensors. \vspace{2mm}

\textit{Proof} of Theorem \ref{BettiNumberEstimate}. By replacing $M$ with its orientation double cover, if necessary, it may be assumed that $M$ is orientable. Due to Poincar\'e duality it suffices to consider $p \leq \floor{\frac{n}{2}}.$ Let $\omega$ be a $p$-form. Lemma \ref{EstimatesForTensors} and proposition \ref{NormsOfHats} imply
\begin{equation*}
| L \omega |^2 \leq p | \omega |^2 |L|^2 = \frac{1}{n-p} | \hat{\omega} |^2 |L|^2
\end{equation*}
for all $L \in \mathfrak{so}(TM).$ 

If the eigenvalues of the curvature operator satisfy $\frac{1}{n-p}(\lambda_{1} + \ldots + \lambda_{n-p}) \geq \kappa,$ then lemma \ref{GeneralBochnerTermEstimate} yields
\begin{equation*}
g( \mathfrak{R}( \hat{w}), \hat{\omega} ) \geq \kappa | \hat{\omega} |^2 = \kappa p(n-p) | \omega |^2.
\end{equation*}

An application of the Bochner technique as in theorem \ref{BochnerTechniqueMachinery} to the Hodge Laplacian completes the proof. \hfill $\Box$ 

\vspace{2mm}

\textit{Proof} of Theorem \ref{GeneralizationTachibana}. Recall from example \ref{ExamplesLichnerowiczLaplacians} (c) that the curvature tensor of an Einstein manifold is harmonic. Hence it satisfies the Bochner formula
\begin{equation*}
\nabla^{*} \nabla \Rm + \frac{1}{2} \Ric(\Rm) = 0.
\end{equation*}

Moreover, since $\mathring{\Ric} = 0$, proposition  \ref{NormsOfHats} shows that $| \widehat{\Rm} |^2 = 4(n-1) | \mathring{\Rm} |^2$ and lemma \ref{EstimatesForTensors} implies
\begin{align*}
|L \Rm |^2 \leq 8 | \Rmzero |^2 |L|^2 = \frac{2}{n-1} | \widehat{\Rm} |^2 | L|^2
\end{align*}
for all $L \in \mathfrak{so}(TM).$ 

By assumption the eigenvalues of the curvature operator satisfy $\lambda_1 + \ldots + \lambda_{\floor{\frac{n-1}{2}}} \geq 0$ and thus lemma \ref{GeneralBochnerTermEstimate} implies 
\begin{align*}
g( \mathfrak{R}( \widehat{\Rm}), \widehat{\Rm}) \geq 0.
\end{align*}

An application of the Bochner technique as in theorem \ref{BochnerTechniqueMachinery} shows that $\Rm$ is parallel. Moreover, if $\lambda_1 + \ldots + \lambda_{\floor{\frac{n-1}{2}}} > 0,$ then $| \widehat{\Rm} |^2 = 0$ and thus $\Rm$ has constant sectional curvature due to proposition \ref{NormsOfHats}. This shows the claim in dimensions $n \geq 5.$

In dimension $n=4,$ choose a Singer-Thorpe basis $\lbrace \Xi_{\alpha} \rbrace$ as in remark \ref{SingerThorpeBasis}. Propositions \ref{ComputationBochnerTermWithEigenvalues} and \ref{LRinSingerThorpeBasis} show that 
\begin{align*}
g( \mathfrak{R}(\widehat{\Rm}), \widehat{\Rm}) = \sum \lambda_{\alpha} | \Xi_{\alpha} \Rm |^2 
= & \ 16 \left\lbrace \lambda_1 (\lambda_2 - \lambda_3)^2 + \lambda_2 (\lambda_1 - \lambda_3)^2 + \lambda_3 (\lambda_1 - \lambda_2)^2 \right. \\
& \ \ \ \ \ \left.  + \lambda_4 (\lambda_5 - \lambda_6)^2 + \lambda_5 (\lambda_4 - \lambda_6)^2 + \lambda_6 (\lambda_4 - \lambda_5)^2 \right\rbrace.
\end{align*}

Suppose that $\mathfrak{R}$ is $2$-nonnegative. After relabeling the basis, if needed, it may be assumed that
\begin{align*}
\lambda_1 + \lambda_2 \geq 0, \ \lambda_1 \leq \lambda_2 \leq \lambda_3, \ \lambda_4, \lambda_5, \lambda_6 \geq 0.
\end{align*}

Notice that $\lambda_2$ might not be the second smallest eigenvalue, but these conditions already imply
\begin{align*}
g( \mathfrak{R}(\widehat{\Rm}), \widehat{\Rm}) \geq 16 \left\lbrace (\lambda_1 + \lambda_2)( \lambda_1- \lambda_3)^2 + \lambda_3 (\lambda_1 - \lambda_2)^2 \right\rbrace  \geq 0.
\end{align*}

Furthermore, if $\mathfrak{R}$ is $2$-positive and $g( \mathfrak{R}(\widehat{\Rm}), \widehat{\Rm})=0,$ then the first Bianchi identity $\lambda_1 + \lambda_2 + \lambda_3 = \lambda_4 + \lambda_5 + \lambda_6$ implies that $\mathfrak{R}$ is a homothety. Hence $(M,g)$ has constant curvature at every tangent space and Schur's lemma implies the claim.
 \hfill $\Box$

\begin{remarkroman} (a) Recall that every irreducible Riemannian manifold with parallel Ricci tensor is Einstein, and hence Theorem \ref{GeneralizationTachibana} applies.

(b) Define $\mathfrak{R} \colon \Lambda^2 \R^4 \to \Lambda^2 \R^4$ on a Singer-Thorpe basis $\lbrace \Xi_{\alpha} \rbrace$ by $\mathfrak{R}( \Xi_{\alpha} ) = \lambda_{\alpha} \Xi_{\alpha}$ and set $\lambda_1 = \lambda_2 = - \lambda,$ $\lambda_3 = 8 \lambda,$  $\lambda_4 = \lambda_5 = \lambda_6 =2 \lambda$ for some $\lambda>0.$ It follows that $\mathfrak{R}$ is Einstein, $3$-nonnegative, and satisfies the first Bianchi identity due to remark \ref{SingerThorpeBasis}. However, notice that $g( \mathfrak{R}(\hat{R}), \hat{R}) < 0.$
\end{remarkroman}

A variation of the proof of Theorem \ref{GeneralizationTachibana} yields a rigidity result for Weyl tensors.

\begin{proposition}
\label{BochnerFormulaForWeylCurvature}
Let $(M,g)$ be a Riemannian manifold. If the Weyl curvature $W$ is divergence free, then $W$ satisfies the second Bianchi identity and 
\begin{align*}
\nabla^{*} \nabla W + \frac{1}{2} \Ric(W) = 0.
\end{align*}
\end{proposition}
\begin{proof}
These are known facts. A proof is included for convenience of the reader. Let 
\begin{align*}
P 
= - \frac{\scal}{2(n-1)(n-2)} g + \frac{1}{n-2} \Ric
\end{align*}
denote the Schouten tensor. Notice that $\Rm = P \owedge g + W$. The contracted second Bianchi identity implies $-2(n-1) (\nabla^{*} P) (X)  =  d \scal (X)$ and hence 
\begin{align*}
(\nabla_Y P)(X,V) - (\nabla_X P)(Y,V) = \frac{1}{n-3}(\nabla^{*}W)(V,X,Y).
\end{align*}
By combining expressions of this type into formulae like
\begin{align*}
(\nabla^{*}W) \left( g(Z,U)W, X, Y \right) - (\nabla^{*}W) \left( g(Z,W)U, X, Y \right) = (\nabla^{*}W) \left( (U \wedge W)Z, X, Y \right),
\end{align*}
one obtains with the second Bianchi identity
\begin{align*}
0 = & \ \left( \nabla_X \Rm \right)(Y,Z,U,W) + \left( \nabla_ Y\Rm \right)(Z, X, U,W) + \left( \nabla_Z \Rm \right)(X,Y,U,W) \\
= & \ \left( (\nabla_X P) \owedge g \right)(Y,Z,U,W) + \left( (\nabla_ Y  P) \owedge g \right)(Z, X, U,W) + \left( (\nabla_Z P) \owedge g \right)(X,Y,U,W) \\
 & \ + ( \nabla_X W )(Y,Z,U,W) + ( \nabla_Y W )(Z,X,U,W) + ( \nabla_Z W )(X,Y,U,W) \\
 = & \ \frac{1}{n-3} \left\lbrace ( \nabla^{*}W ) \left( (U \wedge W)Z, X, Y \right) + ( \nabla^{*}W ) \left( (U \wedge W)X, Y, Z \right) \right. \\
 & \hspace{14mm} \left. + ( \nabla^{*}W ) \left( (U \wedge W)Y, Z, X \right) \right\rbrace \\
 & \ + ( \nabla_X W )(Y,Z,U,W) + ( \nabla_Y W )(Z,X,U,W) + ( \nabla_Z W )(X,Y,U,W).
\end{align*}
Thus $\nabla^{*}W=0$ implies that $W$ satisfies the second Bianchi identity. In this case, the Bochner formula for $W$ is a consequence of  remark \ref{ExamplesLichnerowiczLaplacians} (c), see \cite[Theorem 9.4.2]{PetersenRiemGeom} for an explicit calculation.
\end{proof}

Notice that the formulae above and 
\begin{align*}
(\nabla^{*} \Rm)(Z,X,Y)  = & \ (\nabla_Y \Ric)(X,Z) - (\nabla_X \Ric)(Y,Z) \\
= & \ - (\nabla^{*}P)(Y) g(X,Z) + (\nabla^{*}P)(X)g(Y,Z) + \frac{n-2}{n-3} \nabla^{*}W(Z,X,Y)
\end{align*} 
imply the well known fact that the curvature tensor of a Riemannian manifold is divergence free if and only if the scalar curvature is constant and the Weyl curvature is divergence free, cf. remark \ref{ExamplesLichnerowiczLaplacians} (c).

The following corollary shows that a result similar to Theorem \ref{GeneralizationTachibana} applies to divergence free Weyl tensors.

\begin{corollary}
Let $(M,g)$ be a closed connected $n$-dimensional Riemannian manifold. Suppose that the Weyl curvature satisfies $\nabla^{*} W = 0.$ If the eigenvalues $\lambda_1 \leq \ldots \leq \lambda_{\binom{n}{2}}$ of the curvature operator satisfy 
\begin{align*}
 \lambda_1 + \ldots + \lambda_{\floor{\frac{n-1}{2}}} \geq 0   \ \text{for} \ n \geq 4,
\end{align*}
then the Weyl curvature $W$ is parallel. Moreover, if the inequality is strict, then $(M,g)$ is conformally flat.
\label{BochnerForWeyl}
\end{corollary}
\begin{proof}
Lemma \ref{EstimatesForTensors} and proposition \ref{NormsOfHats} imply that 
\begin{align*}
|L W |^2 \leq 8 | W |^2 |L|^2 = \frac{2}{n-1} | \widehat{W} |^2 | L|^2
\end{align*}
for all $L \in \mathfrak{so}(TM).$ Thus $ \lambda_1 + \ldots + \lambda_{\floor{\frac{n-1}{2}}} \geq 0$ implies that $g(\mathfrak{R}(\widehat{W}),\widehat{W}) \geq 0$ due to lemma \ref{GeneralBochnerTermEstimate}. An application of the Bochner technique as in theorem \ref{BochnerTechniqueMachinery} to the Bochner formula in proposition \ref{BochnerFormulaForWeylCurvature} implies the claim.
\end{proof}

\begin{exampleroman}
The manifold $\R P^n \# \overline{\R P^n}$ has a conformally flat metric with constant positive scalar curvature and non-parallel Ricci tensor:

Consider the warped product $I \times S^{n-1}$ with metric $dr^2 + \rho^2 ds_{n-1}^2.$ Its scalar curvature is given by
\begin{align*}
\scal = - 2(n-1) \frac{\ddot{\rho}}{\rho} + (n-2)(n-1) \frac{1- \dot{\rho}^2}{\rho^2}.
\end{align*}
To obtain constant scalar curvature $\scal = 2(n-1)$, set $x=\rho$, $y=\dot{\rho}$ and consider the ODE
\begin{align*}
\dot{x} & = y, \\
\dot{y} & = - \frac{1}{x} \left( x^2 + \frac{n-2}{2} y^2 - \frac{n-2}{2} \right). 
\end{align*}
Notice that if $(x(t),y(t))$ is a solution, then $(x(-t),-y(-t))$ is also a solution. This implies that the fixed point $(\sqrt{\frac{n-2}{2}},0),$ which is a linear center, is in fact a center for the non-linear ODE. It follows that there is a trajectory in the first quadrant which starts at $(x_0,0)$ with $0<x_0^2<\frac{n-2}{2}$ and which eventually crosses the $x$-axis at $(x_1,0)$ with $x_1^2>\frac{n-2}{2}$. This solution induces $\rho$ with $(\rho(0),\dot{\rho}(0))=(x_0,0)$ and $(\rho(T),\dot{\rho}(T))=(x_1,0)$ as required. 
\end{exampleroman}

In the case of symmetric $(0,2)$-tensors it follows as before that $g( \mathfrak{R}(\hat{h}), \hat{h}) \geq \kappa |\hat{h}|^2$ provided that $\lambda_1 + \ldots + \lambda_{\floor{\frac{n}{2}}} \geq \kappa \floor{\frac{n}{2}}.$ In fact, the following proposition shows that the curvature term $g( \mathfrak{R}(\hat{h}), \hat{h})$ can be controlled by sums of $\floor{\frac{n}{2}}$ complex sectional curvatures. This is to be expected given the previous results of Berger \cite{BergerCompactEinstein}, \cite{BergerKaehlerEinstein} and Simons \cite{SimonsMinimalVarieties} for symmetric $(0,2)$-tensors, Micallef-Wang \cite{MicallefWangNIC} for $2$-forms and the first author's \cite{PetersenRiemGeom} combined proof; see also the related work of Bettiol-Mendes \cite{BettiolMendesWeitzenboeck}.

\begin{proposition}
Let $\mathfrak{R} \colon \Lambda^2V \to \Lambda^2V$ be an algebraic curvature operator so that for every orthonormal basis $e_1, \ldots, e_n$ for $V \otimes \C$ the sum of any $\floor{\frac{n}{2}}$ complex sectional curvatures of the form $K_{ij}^{\C} = \mathfrak{R} \left(e_i \wedge e_j, \overline{e_i \wedge e_j} \right)$, $i<j$, is nonnegative.

If $H \colon V \to V$ is normal and $h$ denotes the associated $(0,2)$-tensor, then $g( \mathfrak{R}( \hat{h}), \hat{h} ) \geq 0$. 
\label{BochnerZeroTwoTensors}
\end{proposition}

\begin{proof}
There is an orthonormal basis $e_1, \ldots, e_n$ for $V \otimes \C$ such that $H(e_i) = h_i e_i$ for $i=1, \ldots, n$. It follows as in \cite[proposition 9.4.12]{PetersenRiemGeom} that
\begin{align*}
g( \mathfrak{R}( \hat{h}), \hat{h} )  = \sum_{i,j} g \left( \mathfrak{R}(\hat{h}(e_i, e_j)), \overline{ \hat{h}( e_i, e_j) } \right)  = \sum_{i, j} \left| h_i-\bar{h}_j \right|^2 g \left( \mathfrak{R}(e_i \wedge e_j), \overline{e_i \wedge e_j} \right).
\end{align*}

For notational simplicity write $K_{ij} = g \left( \mathfrak{R}(e_i \wedge e_j), \overline{e_i \wedge e_j} \right).$ By assumption there are at most $\floor{\frac{n}{2}}$ curvatures with $K_{ij} \leq 0$. Moreover, if there are $\floor{\frac{n}{2}}$ curvatures $K_{ij} \leq 0$, then they are all zero. It may be assumed that there is at least one $K_{ij} \leq 0$. More precisely, suppose are $m$ curvatures $K_{ij} \leq 0$ for some $m \in \lbrace 1, \ldots, \floor{\frac{n}{2}} -1 \rbrace.$ The indices can be rearranged, if necessary, so that $\left| h_1 - \bar{h}_n \right|$ is maximal among all terms $\left| h_i - \bar{h}_j \right|$ with $K_{ij} \leq 0$.  Set
\begin{align*}
a & = \left| \left\lbrace i \in \lbrace 2, \ldots, n-1 \rbrace \ \vert \ K_{1i} \leq 0 \right\rbrace \right|, \\
b & = \left| \left\lbrace j \in \lbrace 2, \ldots, n-1 \rbrace \ \vert \ K_{jn} \leq 0 \right\rbrace \right|.
\end{align*}
Notice that by assumption $a+b \leq m-1$ and that there are 
\begin{align*}
& n-2-a \ \text{ values with } \ K_{1i} > 0, \\
& n-2-b \ \text{ values with } \ K_{jn} > 0.
\end{align*}

This implies that there are at least $l =n-2-(a+b)$ indices $i_1, \ldots, i_l \in \lbrace 2, \ldots, n-1 \rbrace$ such that $K_{1 i_j} > 0$ and $K_{i_j n} >0.$ Notice that $l \geq n-m-1$ and thus
\begin{align*}
\sum_{i<j} \left|h_i - \bar{h}_j \right|^2 K_{ij} \geq & 
\sum_{\alpha \in \lbrace i_1, \ldots, i_l \rbrace} \left\lbrace \left| h_1- \bar{h}_{\alpha} \right|^2 K_{1{\alpha}} + \left| h_{\alpha}-\bar{h}_n \right|^2 K_{{\alpha}n} \right\rbrace
+ \sum_{\substack{i<j \\ K_{ij} \leq 0}} \left| h_i - \bar{h}_j \right|^2 K_{ij} \\
\geq & \  \sum_{\alpha \in \lbrace i_1, \ldots, i_l \rbrace} \frac{1}{2} \left|h_1 - \bar{h}_n \right|^2 \min \left\lbrace K_{1 \alpha}, K_{\alpha n} \right\rbrace 
 + \left| h_1 - \bar{h}_n \right|^2 \sum_{\substack{i<j \\ K_{ij} \leq 0}} K_{ij} \\
\geq & \  \left|h_1 - \bar{h}_n \right|^2 \sum_{j = 1}^{\floor{\frac{n-m-1}{2}}}  \min \left\lbrace K_{1 i_j}, K_{i_j n}, K_{1 i_{j+\floor{\frac{n-m-1}{2}}}}, K_{i_{j+\floor{\frac{n-m-1}{2}}} n} \right\rbrace  \\ 
& \ + \left|h_1 - \bar{h}_n \right|^2  \sum_{\substack{i<j \\ K_{ij} \leq 0}} K_{ij}.
\end{align*} 

The last line yields a sum over $\floor{\frac{n-m-1}{2}} + m \geq \floor{\frac{n}{2}}$ curvatures, only $m \leq \floor{\frac{n}{2}}-1$ of which are nonpositive. Thus, by assumption, this sum is nonnegative.
\end{proof}

\begin{remarkroman}
If $h$ is a symmetric $(0,2)$-tensor, then the curvature term $g( \mathfrak{R}( \hat{h}), \hat{h} )$ is controlled by sums of $\floor{\frac{n}{2}}$ sectional curvatures. 

If $h$ is a $2$-form, then $g( \mathfrak{R}( \hat{h}), \hat{h} )$ can be controlled by sums of $\frac{n}{2}$ isotropic curvatures if $n$ is even and by sums of $\frac{n-1}{2}$ isotropic curvatures of $V \times \R$ if $n$ is odd.
\end{remarkroman}

\begin{remarkroman}
The assumptions in proposition \ref{BochnerZeroTwoTensors} cannot be weakened to sums of more than $\floor{\frac{n}{2}}$ curvatures: If the eigenvalues of $H$ satisfy 
\begin{align*}
h_1 = - \lambda, \ h_2 = \ldots = h_{n-1} = 0 \ \text{and} \ h_n= \lambda \ \text{for some} \ \lambda>0
\end{align*}
and the complex sectional curvatures satisfy 
\begin{align*}
K_{1n}<0 \ \text{and} \ K_{12} = \ldots = K_{1 n-1} = K_{2n} = \ldots = K_{n-1 n} = K > 0,
\end{align*}
then the curvature term in proposition \ref{BochnerZeroTwoTensors} satisfies
\begin{align*}
g(\mathfrak{R}( \hat{h}), \hat{h}) = 2 \sum_{i<j} \left|h_i - \bar{h}_j \right|^2 K_{ij} = 8 \lambda^2  \left\lbrace K_{1n} + \frac{n-2}{2} K \right\rbrace. 
\end{align*}

Notice that every $\mathfrak{R} \in \operatorname{Sym}^2( \Lambda^2V)$ with an eigenbasis consisting of decomposable elements automatically satisfies the first Bianchi identity. In particular, there is an algebraic curvature operator $\mathfrak{R}$ so that every sum $\sum K_{ij}$ over $\floor{\frac{n}{2}}+1$ of its sectional curvatures is positive but $g(\mathfrak{R}( \hat{h}), \hat{h})<0.$
\end{remarkroman}

\section{Examples}
\label{ExampleSection} 
Doubly warped product metrics on $S^n$ with special curvature properties, separating the curvature condition $\lambda_1 + \ldots + \lambda_{n-p} >0$ of Theorem \ref{BettiNumberEstimate} from known Ricci flow invariant curvature conditions, are discussed in example \ref{DoublyWarpedProducts}.

Examples \ref{OptimalitySymmetricTensors}, \ref{OptimalityTwoForms} and \ref{OptimalityPForms}, and \ref{OptimalityCurvTensors} show that the estimates in lemma \ref{EstimatesForTensors} are sharp in the cases of symmetric $(0,2)$-tensors, forms and algebraic curvature tensors, respectively.

An $(n-1)$-nonnegative algebraic curvature operator $\mathfrak{R} \colon \Lambda^2 \R^n \to \Lambda^2 \R^n$ and a $2$-form $\omega$ such that $g( \mathfrak{R}( \hat{\omega}), \hat{\omega}) <0$ are constructed in example \ref{NegativeBochnerTwoForms}.

Examples \ref{CurvatureSphereTimesEuclideanSpace} and \ref{CurvatureProductsOfSpheresTimesEuclideanSpace} compute $| \hat{R}|^2$ for $S^p \times \R^{n-p}$ and $S^2 \times \ldots \times S^2 \times \R^{n-2k}$. This will be used to give a different proof of the formula $| \hat{R} |^2 = 4 (n-1) | \mathring{R} |^2  - 2 | \Riczero |^2 $ of proposition \ref{NormsOfHats}.

\begin{exampleroman}
For $p, q \geq 2$ consider $S^{p+q+1}$ as the doubly warped product
\begin{align*}
\left( \ [0, \pi /2] \times S^p \times S^q, \ dr^2 + \phi^2 ds_p^2 + \psi^2 ds_q^2 \ \right)
\end{align*} 
where $\phi(r) = \sin(r)$ and $\psi(r) = \cos(r)$. If $X$ is tangent to $S^p$ then the curvature operator satisfies $\mathfrak{R}( \partial_r \wedge X) = - \frac{\phi{''}}{\phi} \partial_r \wedge X$ and no other eigenvalue depends on $\phi{''}.$ Thus there is a small $\mathcal{C}^1$-perturbation of $\phi$ so that the eigenvalue $- \frac{\phi{''}}{\phi}$ can be arranged to have an arbitrary negative minimum at some $r_0 \in (0,\pi/2),$ while all other eigenvalues remain close to one. This procedure yields metrics on $S^{p+q+1}$ of the following types:

(a) Metrics that have positive isotropic curvature but do not induce positive isotropic curvature on $S^{p+q+1} \times \R$, and have $(p+1)$-positive curvature operator but do not have $p$-positive curvature operator.

(b) Metrics with negative isotropic curvatures at some tangent space whose curvature operator is $k$-positive but not $(k-1)$-positive for $k=5, \ldots, pq + p +q +1.$
\label{DoublyWarpedProducts}

In particular, Brendle's \cite{BrendleConvergenceInHihgerDimensions} convergence theorem for the Ricci flow, Micallef-Moore's \cite{MicallefMoorePIC} theorem on simply connected manifolds with positive isotropic curvature, 
and Theorem \ref{BettiNumberEstimate} indeed make different assumptions on curvature.
\end{exampleroman}

\begin{exampleroman} Consider the symmetric $(0,2)$-tensor $h = e^1 \otimes e^1 - e^2 \otimes e^2$ and the bivector $L= e_1 \wedge e_2$. It follows that $Lh = -2 (e^1 \otimes e^2 + e^2 \otimes e^1)$ and $| Lh |^2 = 4 |h|^2 |L|^2.$ In particular, the estimates for symmetric $(0,2)$-tensors in proposition \ref{NormActionOnSymTensors} and lemma \ref{EstimatesForTensors} are optimal.
\label{OptimalitySymmetricTensors}
\end{exampleroman}

\begin{exampleroman}
Consider the $2$-forms $\omega_1 = e^1 \wedge e^3 - e^2 \wedge e^4$ and $\omega_2= e^1 \wedge e^4 + e^2 \wedge e^3$ and the bivector $L = e_1 \wedge e_2 + e_3 \wedge e_4.$ It follows that $L \omega_1 = -2 \omega_2$ and $L \omega_2 = 2 \omega_1$. In particular, $|L \omega_1 |^2 = |L\omega_2|^2 = 8,$ $| \omega_1 |^2= | \omega_2 |^2 =2$ and  $|L|^2=2.$ Thus the estimate in lemma \ref{EstimatesForTensors} is optimal for $2$-forms.
\label{OptimalityTwoForms}
\end{exampleroman}

\begin{exampleroman}
For $p$-forms on $\R^{2p}$ consider $L=e_1 \wedge e_2 + \ldots + e_{2p-1} \wedge e_{2p}$. There are $2^p$ forms $\omega$ of the form $e^I = e^{i_1} \wedge \ldots \wedge e^{i_p}$ with $i_1 < \ldots < i_p$ such that $L \omega$ is a linear combination of exactly $p$ forms $e^I$. Notice from the proof of lemma \ref{EstimatesForTensors} that this happens precisely when $i_1 \in \lbrace 1,2 \rbrace, \ldots, i_p \in \lbrace 2p-1, 2p\rbrace.$ The span of these $e^I$ is a subspace which is invariant under $L$. Furthermore, there is a choice of $\alpha_I, \beta_I \in \lbrace \pm 1 \rbrace$ such that $L  \sum \alpha_I e^I  = \pm p \sum \beta_I e^I.$ The signs can be predicted in the following way:

The basis elements will be grouped into $p+1$ groups $B_0, \ldots, B_p$ where $B_k$ consists of $\binom{p}{k}$ basis elements. The coefficients of the basis elements in each group will have the same sign but the coefficients of the basis elements in $B_k$ and $B_{k+2}$ must have opposite signs. Set  $B_0 = \lbrace e^1 \wedge e^3 \wedge \ldots \wedge e^{2p-1} \rbrace.$ Suppose that $B_0, \ldots, B_k$ have already been constructed. Apply $L$ to the elements in $B_k$. This produces $\binom{p}{k+1}$ basis elements which have not occurred in $B_0, \ldots, B_k.$ These elements form $B_{k+1}.$ Note that, e.g., $B_p = \lbrace e_2 \wedge e_4 \wedge \ldots \wedge e_{2p} \rbrace$. Define 
\begin{align*}
\omega_1 & = + \sum_{B_0} e^I - \sum_{B_2} e^I + \sum_{B_4} e^I - \ldots , \\
\omega_2 & = + \sum_{B_1} e^I - \sum_{B_3} e^I + \sum_{B_5} e^I - \ldots.
\end{align*}
It follows that $L \omega_1 = -p \omega_2$ and $L \omega_2 = p \omega_1.$ Notice that $\Omega = \omega_1 \pm \omega_2$ indeed uses $2^p$ basis elements.

In the case $p=3$ one obtains
\begin{align*}
\omega_1 & = e^{1} \wedge e^3 \wedge e^5 - e^{1} \wedge e^4 \wedge e^6 - e^{2} \wedge e^3 \wedge e^6 - e^{2} \wedge e^4 \wedge e^5, \\
\omega_2 & = e^{1} \wedge e^3 \wedge e^6 + e^{1} \wedge e^4 \wedge e^5 + e^{2} \wedge e^3 \wedge e^5 - e^{2}\wedge e^4 \wedge e^6.
\end{align*}

For dimensions $n \geq 2p$ notice that $\Lambda^2 (\R^{2p})^{*} \subseteq \Lambda^2 (\R^{n})^{*}$. This shows that the estimate of lemma \ref{EstimatesForTensors} is optimal for $p$-forms.
\label{OptimalityPForms}
\end{exampleroman}

\begin{exampleroman}
Let $\lambda > 0$ and let $\Xi_1, \ldots, \Xi_6$ be the Singer-Thorpe basis defined in remark \ref{LRinSingerThorpeBasis}. Consider the algebraic curvature operator $\mathfrak{R}$ on $\Lambda^2 \R^4$ given by
\begin{align*}
\mathfrak{R}( \Xi_1 ) & = - \lambda \Xi_1, \\
\mathfrak{R}( \Xi_2 ) & = \lambda \Xi_2, \\
\mathfrak{R}( \Xi_3 ) & = 3\lambda \Xi_3, \\
\mathfrak{R}( \Xi_i ) & = \lambda \Xi_i \ \text{ for } \ i=4,5,6.
\end{align*}
Notice that $\mathfrak{R}$ is $2$-nonnegative and Einstein, has $| \mathring{R} |^2 = 8 \lambda^2$ and satisfies
\begin{align*}
| \Xi_1 R |^2 & = 2 | \mathring{R} |^2, \\
| \Xi_2 R |^2 & = 8 | \mathring{R} |^2, \\ 
| \Xi_3 R |^2 & = 2 | \mathring{R} |^2, \\
\Xi_i R & =  0 \ \text{ for } \ i=4,5,6
\end{align*}
due to proposition \ref{LRinSingerThorpeBasis}. 

In particular, this example achieves equality in the estimate for curvature tensors in lemma \ref{EstimatesForTensors}. 

This observation also implies that $\mathfrak{R}$ indeed satisfies the first Bianchi identity: Recall that a tensor $T \in \operatorname{Sym}^2(\Lambda^2 \R^n)$ satisfies the Bianchi identity if and only if it is orthogonal to $\Lambda^4 \R^n.$ Due to lemma \ref{EstimatesForTensors} the orthogonal projection of $\mathfrak{R}$ onto $\Lambda^4 \R^4$ cannot achieve equality in $|L R|^2 \leq 8 |R|^2 |L|^2$ for any $L \in \mathfrak{so}(\R^4).$ Thus, since $\mathfrak{R}$ does achieve equality, its orthogonal projection onto $\Lambda^4 \R^4$ must vanish. 

Due to corollary \ref{LRmWithVanishingWeyl} a similar argument shows that every curvature tensor which maximizes $|L \Rm|^2 \leq 8 |\mathring{\Rm}|^2 |L|^2$ for some $L \in \mathfrak{so}(\R^n)$ must be Einstein and cannot have vanishing Weyl curvature.
\label{OptimalityCurvTensors}
\end{exampleroman}

\begin{exampleroman}
Let $\Xi_1, \ldots, \Xi_6$ be the Singer-Thorpe basis defined in remark \ref{SingerThorpeBasis} and set $\omega=e^1 \wedge e^4 + e^2 \wedge e^3.$ It follows that $| \Xi_1 \omega |^2 = | \Xi_2 \omega |^2 = 2 |\omega|^2$ and $\Xi_i \omega=0$ for $i=3, \ldots, 6.$ Recall that any operator $\mathfrak{R} \colon \Lambda^2 \R^4 \to \Lambda^2 \R^4$ with $\mathfrak{R}(\Xi_{\alpha}) = \lambda_{\alpha} \Xi_{\alpha}$ for $\alpha=1, \ldots, 6$ satisfies the first Bianchi identity if and only if $\lambda_1 +\lambda_2 + \lambda_3 = \lambda_4 + \lambda_5 + \lambda_6.$ Furthermore, proposition \ref{ComputationBochnerTermWithEigenvalues} implies that 
\begin{align*}
g( \mathfrak{R}(\hat{\omega}), \hat{\omega} ) = 2 ( \lambda_1 + \lambda_2 ) | \omega |^2.
\end{align*}

In particular, the above curvature term for $\C P^2$ with the Fubini Study metric vanishes on the associated K{\"a}hler form $\omega_{FS},$ see also remark \ref{LREstimateOptimalForCPTwo}. In fact, the example of $\C P^2$ shows that Theorem \ref{BettiNumberEstimate} fails if its assumptions are weakened to $3$-positive curvature operators in dimension $n=4.$ 

Furthermore, setting $\lambda_1 = \lambda_2 = - \lambda,$ $\lambda_3 = 8 \lambda$ and $\lambda_4=\lambda_5=\lambda_6=2\lambda$ for some $\lambda >0$ yields a $3$-nonnegative curvature operator $\mathfrak{R} \colon \Lambda^2 \R^4 \to \Lambda^2 \R^4$ with $g( \mathfrak{R}(\hat{\omega}), \hat{\omega} ) = -4 \lambda | \omega |^2 <0.$ 

More generally, for $n \geq 4,$ an $(n-1)$-nonnegative curvature operator $\mathfrak{R} \colon \Lambda^2 \R^n \to \Lambda^2 \R^n$ with $g( \mathfrak{R}(\hat{\omega}), \hat{\omega} ) <0$ is given as follows:

Extend the Singer-Thorpe basis $\Xi_1, \ldots, \Xi_6$ for $\Lambda^2 \R^4$ to a basis $\lbrace \Xi_{\alpha} \rbrace$ for $\Lambda^2 \R^n$ by including the forms $e_i \wedge e_j$ for $i \in \lbrace 1, \ldots, 4 \rbrace$, $j \in \lbrace 5, \ldots, n \rbrace$ and $i, j \in \lbrace 5, \ldots, n \rbrace$ with $i < j.$ It follows that 
\begin{align*}
| (e_i \wedge e_j) \omega |^2 = \begin{cases}
\frac{1}{2} |\omega|^2 & \ \text{if} \ i \in \lbrace 1, \ldots, 4 \rbrace \ \text{and} \ j \in \lbrace 5, \ldots, n \rbrace, \\
0 & \ \text{if} \ i, j \in \lbrace 5, \ldots, n \rbrace.
\end{cases}
\end{align*}

The operator $\mathfrak{R} \colon \Lambda^2 \R^n \to \Lambda^2 \R^n$ defined by $\mathfrak{R}(\Xi_{\alpha}) = \lambda_{\alpha} \Xi_{\alpha}$ for $\alpha=1, \ldots, \binom{n}{2}$ still satisfies the first Bianchi identity if and only if $\lambda_1 +\lambda_2 + \lambda_3 = \lambda_4 + \lambda_5 + \lambda_6$. Pick $\lambda>0$ and set
\begin{align*}
\lambda_1 = \lambda_2 = -(n-3) \lambda, \ \lambda_3=2n \lambda \ \text{and} \ \lambda_4 = \ldots = \lambda_{\binom{n}{2}} = 2\lambda.
\end{align*}
It follows that $\mathfrak{R}$ is an $(n-1)$-nonnegative algebraic curvature operator with 
\begin{align*}
g( \mathfrak{R}(\hat{\omega}), \hat{\omega} ) = \lbrace 2 (\lambda_1 + \lambda_2 ) + 4(n-4) \lambda \rbrace | \omega |^2 = - 4 \lambda | \omega |^2 < 0.
\end{align*}
Thus there also is an $(n-1)$-positive algebraic curvature operator $\widetilde{\mathfrak{R}}$ with $g( \widetilde{\mathfrak{R}}(\hat{\omega}), \hat{\omega} )<0.$ 
\label{NegativeBochnerTwoForms}
\end{exampleroman}

\begin{remarkroman}
Let $\mathfrak{R} \colon \Lambda^2 \R^n \to \Lambda^2 \R^n$ be a self-adjoint operator and $2p \leq n.$ If $\mathfrak{R}$ is $(n-p)$-positive, then lemmas \ref{GeneralBochnerTermEstimate}, \ref{EstimatesForTensors} and proposition \ref{NormsOfHats} show that $g( \mathfrak{R}( \hat{\omega}), \hat{\omega})>0$ for every non-zero $\omega \in \Lambda^p (\R^n)^{*}.$ In particular, $\mathfrak{R}$ does not need to satisfy the first Bianchi identity.

Given the examples of $\omega \in \Lambda^p (\R^{2p})^{*}$ and $\Xi \in \mathfrak{so}(\R^{2p})$ with $|\Xi \omega|^2 = p |\omega|^2$ in example \ref{OptimalityPForms}, for every $n \geq 2p$ it is easy to find a self-adjoint, $(n-p+1)$-positive operator $\mathfrak{R} \colon \Lambda^2 \R^n \to \Lambda^2 \R^n$ with $g( \mathfrak{R}( \hat{\omega}), \hat{\omega})<0$.
\end{remarkroman}

\begin{exampleroman}
The curvature tensor $R$ of $S^p \times \R^{n-p}$ satisfies
\begin{equation*}
| \hat{R} |^2 = 2 (p-1) p (n-p).
\end{equation*}
\label{CurvatureSphereTimesEuclideanSpace}
\end{exampleroman}
\begin{proof}
Let $e_1, \ldots, e_n$ be an orthonormal basis such that $e_1, \ldots, e_p$ correspond to tangent vectors of $S^p$ and notice that the curvature operator $\mathfrak{R}$ of $S^p \times \R^{n-p}$  satisfies
\begin{align*}
\mathfrak{R}(e_i \wedge e_j)  = e_i \wedge e_j \ \text{ and } \ 
\mathfrak{R}(e_i \wedge e_a)  = \mathfrak{R}(e_a \wedge e_b) = 0
\end{align*}
for $i,j \in \lbrace 1, \ldots, p \rbrace$ and $a, b \in \lbrace p+1, \ldots, n \rbrace.$ In the following, consider $i<j$ and $a<b.$

Index the above basis of $\Lambda^2\R^n$ so that the elements $\Xi_x = e_i \wedge e_j$, $\Xi_y = e_i \wedge e_a$ and $\Xi_z = e_a \wedge e_b$ satisfy $x<y<z.$ To calculate $| \hat{R} |^2$ via proposition \ref{NormLRinTermsOfEigenvalues} notice that for $\alpha < \beta$ the term $(\lambda_{\alpha} - \lambda_{\beta} ) g( \Xi_{\alpha}, ( \Xi_{\gamma} ) \Xi_{\beta} )$ can only be non-zero if $\Xi_{\alpha} = e_i \wedge e_j$ and $\Xi_{\beta} = e_k \wedge e_a$ for some $k \in \lbrace 1, \ldots, p \rbrace.$ If $\Xi_{\gamma} = e_{\lambda} \wedge e_{\mu},$ then 
\begin{align*}
g( \Xi_{\alpha}, ( \Xi_{\gamma} ) \Xi_{\beta} ) 
& = g( e_i \wedge e_j, \delta_{\lambda a} e_k \wedge e_{\mu} - \delta_{a \mu} e_k \wedge e_{\lambda} ) \\
& = \delta_{\lambda a} ( \delta_{ik} \delta_{j \mu} - \delta_{i \mu} \delta_{j k} ) - \delta_{a \mu} ( \delta_{ik} \delta_{j \lambda} - \delta_{ i \lambda} \delta_{j k})
\end{align*}
is non-zero for $\lambda < \mu$ only if $k=i$ and $\lambda = j, \mu=a$ or $k =j$ and $\lambda=i, \mu=a.$
This implies
\begin{align*}
| (e_k \wedge e_a) R |^2 
& = 2 \sum_{1 \leq i < j \leq p} \sum_{\substack{l =1, \ldots, p \\ b=p+1, \ldots, n}} g(e_i \wedge e_j, (e_k \wedge e_a) e_l \wedge e_b)^2 \\
& = 2 \sum_{1 \leq i < j \leq p} \left\lbrace  g(e_i \wedge e_j, (e_k \wedge e_a) e_i \wedge e_a)^2 + g(e_i \wedge e_j, (e_k \wedge e_a) e_j \wedge e_a)^2\right\rbrace \\
& = 2 \sum_{\substack{i=1 \\ i \neq k}}^p 1 = 2(p-1)
\end{align*}
for $k \in \lbrace 1, \ldots, p \rbrace$ and $a \in \lbrace p+1, \ldots, n \rbrace$ and therefore
\begin{equation*}
| \hat{R} |^2 = \sum_{\substack{k=1, \ldots, p \\ a=p+1, \ldots, n}} | (e_k \wedge e_a) R |^2 = 2 (p-1) p (n-p)
\end{equation*}
as claimed.
\end{proof}

Example \ref{CurvatureSphereTimesEuclideanSpace} can be used to give an alternative {\em proof} of 
\begin{equation*}
| \hat{R} |^2 = 4 (n-1) | \mathring{R} |^2  - 2 | \Riczero |^2 
\end{equation*}
of proposition \ref{NormsOfHats}: 

Due to the decomposition of $\operatorname{Sym}_B^2(\Lambda^2\R^n)$ into $O(n)$-irreducible orthogonal summands there are constants $a, b, c \in \R$ such that 
\begin{align*}
| \hat{R} |^2 = a \scal^2 + b | \Ric |^2 + c | R |^2
\end{align*}
for all algebraic curvature operators on $\R^n.$ Evaluation on the algebraic curvature tensors of $S^p \times \R^{n-p}$ implies $a=0,$ $b=-2$ and $c=4(n-1).$ 

Notice that $a=0$ is also immediate from the fact that $|LR| = | L \mathring{R}|$ for all $L \in \mathfrak{so}(\R^n)$, and for the same reason $R$ and $\Ric$ can be replaced by $\mathring{R}$ and $\mathring{\Ric},$ respectively. \hfill $\Box$

\begin{exampleroman}
The curvature tensor of $S^2 \times \ldots \times S^2 \times \R^{n-2k}$ satisfies
\begin{align*}
| \hat{R} |^2 = 4k(n-2).
\end{align*}
\label{CurvatureProductsOfSpheresTimesEuclideanSpace}
\end{exampleroman}
\begin{proof}
Let $e_1, \ldots, e_n$ be an orthonormal basis for $\R^n$ such that $e_{2i-1}, e_{2i}$ correspond to tangent vectors of the $i^{\text{th}}$ $S^2$-factor. Thus $e_1 \wedge e_2, \ldots, e_{2k-1} \wedge e_{2k}$ are eigenvectors of the curvature operator $\mathfrak{R}$ with eigenvalue $\lambda = 1$ and the remaining vectors of the basis $e_1 \wedge e_2, \ldots, e_{n-1} \wedge e_n$ for $\Lambda^2 \R^n$ form a basis for the kernel of $\mathfrak{R}.$

Index the above basis $\lbrace \Xi_{\alpha} \rbrace$ for $\Lambda^2 \R^n$  so that $\Xi_i = e_{2i-1} \wedge e_{2i}$ for $i=1, \ldots, k.$ Then for $\alpha < \beta$ the term $(\lambda_{\alpha} - \lambda_{\beta} ) g( \Xi_{\alpha}, ( \Xi_{\gamma} ) \Xi_{\beta} )$ can only be non-zero if $\Xi_{\alpha} = e_{2i-1} \wedge e_{2i}$ for some $i \in \lbrace 1, \ldots, k \rbrace$ and $\Xi_{\beta} = \pm e_{j} \wedge e_{a}$ for $j \in \lbrace 2i-1, 2i \rbrace$ and $a \in \lbrace 1, \ldots, n \rbrace \setminus \lbrace 2i-1, 2i \rbrace.$ In this case $| g( e_{2i-1} \wedge e_{2i}, (\Xi_{\gamma}) e_j \wedge e_a ) | = 1$ if and only if $\Xi_{\gamma} = \pm e_l \wedge e_a$ for $l \in \lbrace 2i-1, 2i \rbrace \setminus \lbrace j \rbrace$.

Thus for $j \in \lbrace 2i-1, 2i \rbrace \subseteq \lbrace 1, \ldots, 2k \rbrace$ and $a \in \lbrace 1, \ldots, n \rbrace \setminus \lbrace 2i-1, 2i \rbrace$  proposition \ref{NormLRinTermsOfEigenvalues} implies 

\begin{align*}
| (e_j \wedge e_a ) R|^2 
& = 2 \sum_{l=2i-1, 2i} \sum_{\substack{b = 1, \ldots, n \\ b \neq 2i-1, 2i }} g(e_{2i-1} \wedge e_{2i}, (e_j \wedge e_a) e_l \wedge e_b )^2 \\
& = 2 \sum_{l=2i-1, 2i} g(e_{2i-1} \wedge e_{2i}, (e_j \wedge e_a) e_l \wedge e_a)^2  = 2
\end{align*}
and hence
\begin{align*}
| \hat{R} |^2 = \sum_{j=1}^{2k} \sum_{ \substack{a = 1, \ldots, n \\ a \neq j, j+ (-1)^{j+1} }} |(e_j \wedge e_a) R|^2 = 4k(n-2).
\end{align*}
\end{proof}




\begin{thebibliography}{Flat}
\bibitem[Ber61a]{BergerTwoFormsVanishCurvOperator}
Marcel Berger, \emph{Sur les vari\'{e}t\'{e}s \`a op\'{e}rateur de courbure
  positif}, C. R. Acad. Sci. Paris \textbf{253} (1961), 2832--2834.

\bibitem[Ber61b]{BergerCompactEinstein}
\bysame, \emph{Sur quelques vari\'{e}t\'{e}s d'{E}instein compactes}, Ann. Mat.
  Pura Appl. (4) \textbf{53} (1961), 89--95.

\bibitem[Ber63]{BergerKaehlerEinstein}
\bysame, \emph{Les vari\'{e}t\'{e}s k\"{a}hl\'{e}riennes compactes d'{E}instein
  de dimension quatre \`a courbure positive}, Tensor (N.S.) \textbf{13} (1963),
  71--74.

\bibitem[B{\'{e}}r88]{BerardBochnerRevisited}
Pierre~H. B{\'{e}}rard, \emph{From vanishing theorems to estimating theorems:
  the {B}ochner technique revisited}, Bull. Amer. Math. Soc. (N.S.) \textbf{19}
  (1988), no.~2, 371--406.

\bibitem[BM20]{BettiolMendesWeitzenboeck}
Renato~G. Bettiol and Ricardo A.~E. Mendes, \emph{{Sectional curvature and
  Weitzenb{\"o}ck formulae}}, to appear in Indiana Univ. Math. J.,
  arXiv:1708.09033 (2020).

\bibitem[Boc46]{BochnerVectoreFieldsAndRic}
S.~Bochner, \emph{Vector fields and {R}icci curvature}, Bull. Amer. Math. Soc.
  \textbf{52} (1946), 776--797.

\bibitem[Bre08]{BrendleConvergenceInHihgerDimensions}
Simon Brendle, \emph{{A general convergence result for the Ricci flow}}, Duke
  Math. J. \textbf{145} (2008), 585--601.

\bibitem[Bre10]{BrendleEinsteinNIC}
Simon Brendle, \emph{Einstein manifolds with nonnegative isotropic curvature
  are locally symmetric}, Duke Math. J. \textbf{151} (2010), no.~1, 1--21.

\bibitem[Bre19]{BrendleRFwithSurgeryPIC}
\bysame, \emph{Ricci flow with surgery on manifolds with positive isotropic
  curvature}, Ann. of Math. (2) \textbf{190} (2019), no.~2, 465--559.

\bibitem[BS08]{BrendleSchoenWeaklyQuarterPinched}
Simon Brendle and Richard~M. Schoen, \emph{Classification of manifolds with
  weakly {$1/4$}-pinched curvatures}, Acta Math. \textbf{200} (2008), no.~1,
  1--13.

\bibitem[BS09]{BrendleSchoenSphereTheorem}
Simon Brendle and Richard Schoen, \emph{{Manifolds with 1/4-pinched curvature
  are space forms}}, J. Amer. Math. Soc. \textbf{22} (2009), no.~1, 287--307.

\bibitem[BW08]{BW2}
Christoph B\"{o}hm and Burkhard Wilking, \emph{{Manifolds with positive
  curvature operators are space forms}}, Ann. of Math. (2) \textbf{167} (2008),
  1079--1097.

\bibitem[Che86]{CheegerVanishingPiecewiseConstCurv}
Jeff Cheeger, \emph{A vanishing theorem for piecewise constant curvature
  spaces}, Curvature and topology of {R}iemannian manifolds ({K}atata, 1985),
  Lecture Notes in Math., vol. 1201, Springer, Berlin, 1986, pp.~33--40.

\bibitem[Che91]{ChenQuarterPinching}
Haiwen Chen, \emph{Pointwise {$\frac14$}-pinched {$4$}-manifolds}, Ann. Global
  Anal. Geom. \textbf{9} (1991), no.~2, 161--176.

\bibitem[CTZ12]{ChenTangZhuClassFourPIC}
Bing-Long Chen, Siu-Hung Tang, and Xi-Ping Zhu, \emph{Complete classification
  of compact four-manifolds with positive isotropic curvature}, J. Differential
  Geom. \textbf{91} (2012), no.~1.

\bibitem[CZ06]{ChenZhuRFwithSurgeryFourPIC}
Bing-Long Chen and Xi-Ping Zhu, \emph{Ricci flow with surgery on four-manifolds
  with positive isotropic curvature}, J. Differential Geom. \textbf{74} (2006),
  no.~2, 177--264.

\bibitem[DN05]{DussanNoronhaNICandPureCurv}
Martha Dussan and Maria~Helena Noronha, \emph{Compact manifolds of nonnegative
  isotropic curvature and pure curvature tensor}, Balkan J. Geom. Appl.
  \textbf{10} (2005), no.~2, 58--66.

\bibitem[Gal81]{GallotSobolevEstimates}
Sylvestre Gallot, \emph{Estim\'{e}es de {S}obolev quantitatives sur les
  vari\'{e}t\'{e}s riemanniennes et applications}, C. R. Acad. Sci. Paris
  S\'{e}r. I Math. \textbf{292} (1981), no.~6, 375--377.

\bibitem[GM75]{GallotMeyerCurvOperatorAndForms}
S.~Gallot and D.~Meyer, \emph{Op\'{e}rateur de courbure et laplacien des formes
  diff\'{e}rentielles d'une vari\'{e}t\'{e} riemannienne}, J. Math. Pures Appl.
  (9) \textbf{54} (1975), no.~3, 259--284.

\bibitem[Gol98]{GoldbergCurvautreAndHolomogy}
Samuel~I. Goldberg, \emph{Curvature and homology}, Dover Publications, Inc.,
  Mineola, NY, 1998.

\bibitem[Gro81]{GromovCurvDiamBetti}
Michael Gromov, \emph{Curvature, diameter and {B}etti numbers}, Comment. Math.
  Helv. \textbf{56} (1981), no.~2, 179--195.

\bibitem[Ham82]{Hamilton3DimRF}
Richard~S. Hamilton, \emph{{Three-manifolds with positive Ricci curvature}}, J.
  Differential Geom. \textbf{17} (1982), 255--306.

\bibitem[Ham86]{Hamilton4DimRFposCurvOp}
\bysame, \emph{{Four-manifolds with positive curvature operator}}, J.
  Differential Geom. \textbf{24} (1986), 153--179.

\bibitem[Ham97]{HamiltonFourPIC}
\bysame, \emph{Four-manifolds with positive isotropic curvature}, Comm. Anal.
  Geom. \textbf{5} (1997), no.~1, 1--92.

\bibitem[Hoe16]{HoelzelSurgeryStableCurvCond}
Sebastian Hoelzel, \emph{Surgery stable curvature conditions}, Math. Ann.
  \textbf{365} (2016), no.~1-2, 13--47.

\bibitem[Hua19]{HuangManifoldsPIC}
Hong Huang, \emph{{Compact manifolds of dimension $n \geq 12$ with positive
  isotropic curvature}}, arXiv:1909.12265 (2019).

\bibitem[Li80]{LiSobolevConstant}
Peter Li, \emph{On the {S}obolev constant and the {$p$}-spectrum of a compact
  {R}iemannian manifold}, Ann. Sci. \'{E}cole Norm. Sup. (4) \textbf{13}
  (1980), no.~4, 451--468.

\bibitem[Mey71]{DMeyerCurvOpPos}
Daniel Meyer, \emph{Sur les vari\'{e}t\'{e}s riemanniennes \`a op\'{e}rateur de
  courbure positif}, C. R. Acad. Sci. Paris S\'{e}r. A-B \textbf{272} (1971),
  A482--A485.

\bibitem[MM88]{MicallefMoorePIC}
Mario~J. Micallef and John~Douglas Moore, \emph{Minimal two-spheres and the
  topology of manifolds with positive curvature on totally isotropic
  two-planes}, Ann. of Math. (2) \textbf{127} (1988), no.~1, 199--227.

\bibitem[Mok88]{MokUniformizationKaehler}
Ngaiming Mok, \emph{The uniformization theorem for compact {K}\"{a}hler
  manifolds of nonnegative holomorphic bisectional curvature}, J. Differential
  Geom. \textbf{27} (1988), no.~2, 179--214.

\bibitem[MW93]{MicallefWangNIC}
Mario~J. Micallef and McKenzie~Y. Wang, \emph{Metrics with nonnegative
  isotropic curvature}, Duke Math. J. \textbf{72} (1993), no.~3, 649--672.

\bibitem[NW07]{NiWuNonnegativeCurvatureOperator}
Lei Ni and Baoqiang Wu, \emph{Complete manifolds with nonnegative curvature
  operator}, Proc. Amer. Math. Soc. \textbf{135} (2007), no.~9, 3021--3028.

\bibitem[Pet16]{PetersenRiemGeom}
Peter Petersen, \emph{{Riemannian Geometry}}, third ed., Graduate Texts in
  Mathematics, vol. 171, Springer, 2016.

\bibitem[Poo80]{PoorHolonomyProofPosCurvOperatorThm}
W.~A. Poor, \emph{A holonomy proof of the positive curvature operator theorem},
  Proc. Amer. Math. Soc. \textbf{79} (1980), no.~3, 454--456.

\bibitem[Ses09]{SeshadriNIC}
Harish Seshadri, \emph{Manifolds with nonnegative isotropic curvature}, Comm.
  Anal. Geom. \textbf{17} (2009), no.~4, 621--635.

\bibitem[Sim68]{SimonsMinimalVarieties}
James Simons, \emph{Minimal varieties in riemannian manifolds}, Ann. of Math.
  (2) \textbf{88} (1968), 62--105.

\bibitem[ST69]{SingerThorpeEinstein}
I.~M. Singer and J.~A. Thorpe, \emph{The curvature of {$4$}-dimensional
  {E}instein spaces}, Global {A}nalysis ({P}apers in {H}onor of {K}.
  {K}odaira), Univ. Tokyo Press, Tokyo, 1969, pp.~355--365.

\bibitem[Tac74]{TachibanaPosCurvOperator}
Shun-ichi Tachibana, \emph{A theorem on {R}iemannian manifolds of positive
  curvature operator}, Proc. Japan Acad. \textbf{50} (1974), 301--302.

\bibitem[YB53]{YanoBochnerCurvAndBetti}
K.~Yano and S.~Bochner, \emph{Curvature and {B}etti numbers}, Annals of
  Mathematics Studies, No. 32, Princeton University Press, Princeton, N. J.,
  1953.
\end{thebibliography}

\end{document}